\documentclass[10pt]{article}
\usepackage{epsfig}
\usepackage{amssymb,amsmath,amsthm,url}
\usepackage{multirow}
\usepackage{booktabs}
\usepackage{natbib}
\usepackage{hyperref}
\usepackage{setspace}
\usepackage{dsfont}

\usepackage{times,amsfonts,eucal}
\usepackage{graphicx,ifthen}
\usepackage{wrapfig}
\usepackage{latexsym}
\usepackage{color}
\usepackage{mathrsfs}
\usepackage{bm}
\usepackage{bbm}
\usepackage{srctex}
\usepackage{enumitem}

\usepackage{mathtools}
\DeclarePairedDelimiter{\floor}{\lfloor}{\rfloor}
\DeclarePairedDelimiter\ceil{ \lceil}{ \rceil}
\theoremstyle{plain}
\newtheorem{theorem}{Theorem}[section]
\newtheorem{corollary}[theorem]{Corollary}
\newtheorem{lemma}[theorem]{Lemma}
\newtheorem{proposition}[theorem]{Proposition}

\theoremstyle{definition}

\theoremstyle{remark}

\numberwithin{equation}{section}

\newcommand{\N}{\mathbb{N}}
\newcommand{\R}{\mathbb{R}}

\newcommand{\Z}{\mathbb{Z}}

\newcommand{\p}{\mathbbm{P}}

\newcommand{\cA}{\mathcal{A}}
\newcommand{\cB}{\mathcal{B}}
\newcommand{\cC}{\mathcal{C}}
\newcommand{\cD}{\mathcal{D}}

\newcommand{\cG}{\mathcal{G}}
\newcommand{\cF}{\mathcal{F}}
\newcommand{\cH}{\mathcal{H}}
\newcommand{\cK}{\mathcal{K}}

\newcommand{\cS}{\mathcal{S}}

\newcommand{\cM}{\mathcal{M}}
\newcommand{\cO}{\mathcal{O}}

\newcommand{\fS}{\mathfrak{S}}
\newcommand{\fT}{\mathfrak{T}}
\newcommand{\E}[1]{\mathbb{E}\left [ \, #1 \, \right ]}

\renewcommand{\epsilon}{\varepsilon}
\renewcommand{\phi}{\varphi}
\newcommand{\pspace}{(\Omega,\cA,\p)}
\newcommand{\intd}[1]{\,\mathrm{d}#1}
\newcommand{\norm}[1]{\left\lVert #1 \right\rVert}
\newcommand{\scalar}[2]{\left\langle #1,#2 \right\rangle}
\newcommand{\1}[1]{\,\mathbbm{1}\! \left\{ #1 \right\} }
\newcommand{\hnorm}[1]{ {\left \lVert #1 \right \rVert}_{\cH} }

\newcommand{\Kh}[1]{K_h\left( #1 \right)}

\allowdisplaybreaks
\begin{document}

\title{A Note on Exponential Inequalities in Hilbert Spaces\\
for Spatial Processes with Applications to the\\
Functional Kernel Regression Model
\footnote{This research was supported by the German Research Foundation (DFG), grant number KR 4977/1-1.}
}

\author{
Johannes T. N. Krebs\footnote{Department of Statistics, University of California, Davis, CA, 95616, USA, email: \tt{jkrebs@ucdavis.edu} }\;\footnote{Corresponding author}
}

\date{\today}
\maketitle
\linespread{1.1}
\begin{abstract}
\setlength{\baselineskip}{1.8em}
%% Text of abstract
In this manuscript we present exponential inequalities for spatial lattice processes which take values in a separable Hilbert space and satisfy certain dependence conditions. We consider two types of dependence: spatial data under $\alpha$-mixing conditions and spatial data which satisfies a weak dependence condition introduced by \cite{dedecker2005new}. We demonstrate their usefulness in the functional kernel regression model of \cite{ferraty2004nonparametric} where we study uniform consistency properties of the estimated regression operator on increasing subsets of the underlying function space.\medskip\\
\noindent {\bf Keywords:} Asymptotic inequalities; Functional data; Nonparametric statistics; Spatial Lattice Processes; Strong mixing; Weak dependence measures\\
\noindent {\bf MSC 2010:} Primary:	62G08; 62M40 Secondary:  37A25; 62G20
\end{abstract}

%% main text
This article studies the nonparametric regression problem for spatial functional data. Pioneering work in functional data analysis has been done by \cite{ramsay1997functional} and \cite{bosq_linear_2000}. The last-named was among the first who considered the linear functional autoregressive model and related estimation techniques. Recently, the analysis of spatial data has gained importance in many applications such as image analysis, geophysics astronomy and environmental science. A systematic introduction to random fields is given in \cite{guyon1995random} or in \cite{cressie1993statistics}. In the same time, technological advances make it possible to sample data at high frequencies such that sample data nowadays can be rather considered as a collection of objects in an infinite-dimensional space, the so-called functional data.

In this article, we address one problem related to functional data, more precisely, the estimation of the regression operator in a nonlinear double functional regression model where both the regressor and the predictor are functional and where the data are generated by a spatial lattice process. We do this in the functional kernel regression model of \cite{ferraty2002functional}, \cite{ferraty_nonparametric_2007} and \cite{ferraty2012regression}.

So far, nonparametric regression for finite-dimensional spatial data has been studied in several variants: \cite{li2016nonparametric} studies a wavelet approach. \cite{krebs2017orthogonal} constructs an orthogonal series estimator for spatial data. In particular, the kernel method has been popular for regression problems which involve spatial data, e.g., see \cite{carbon1996kernel}, \cite{tran1990kernel}, \cite{hallin2004local} and \cite{carbon2007kernel}.

Often the dependence within the spatial data or the time series is assumed to satisfy a strong mixing condition, see \cite{bradley2005basicMixing} for an introduction to mixing conditions. \cite{ferraty2004nonparametric}, \cite{delsol2009advances} study the functional regression model for $\alpha$-mixing time series. We generalize their results to $\alpha$-mixing spatial processes in one part of the manuscript.

Unfortunately, many stochastic processes lack certain smoothness conditions and are thus not $\alpha$-mixing, see e.g. \cite{andrews1984non}. So other dependence concepts have been studied as well: \cite{laib2010nonparametric} consider the functional kernel regression model for stationary ergodic data. \cite{hormann2010weakly} study $L^p$-$m$-approximable functional data. An alternative notion of dependence has been proposed by \cite{dedecker2005new}: their definition of the weak dependence coefficient admits to consider only a finite time interval in the future. We continue with this approach and also study the functional kernel regression model for $\cC$-weakly dependent spatial data, see \cite{maume2006exponential} for a similar application to finite-dimensional time series.

\cite{politis1994limit} develop limit theorems for sums of weakly dependent Hilbert space-valued random variables. We give in this article exponential inequalities for Hilbert space-valued spatial data and continue with the investigations of \cite{ferraty2004nonparametric} and \cite{ferraty2012regression}: we study the uniform $a.s.$-convergence of the kernel regression estimator on increasing subsets of an infinite-dimensional function space.

This paper is organized as follows: we introduce in Section~\ref{Section_DefinitionsNotation} two selected dependence concepts for spatial data. We study exponential inequalities for $\alpha$-mixing Hilbert space-valued spatial processes in Section~\ref{Section_ExponentialInequalities}. Moreover, we give exponential inequalities for $\cC$-weakly dependent Hilbert space-valued spatial processes in Section~\ref{Section_ExpInequalitiesPhiMixing}. In the last Section~\ref{Section_Application}, we apply the inequalities in the functional kernel regression framework of \cite{ferraty2004nonparametric}.

\section{Two dependence concepts for spatial processes}\label{Section_DefinitionsNotation}
Let $\pspace$ be a probability space, $(T,\fT)$ be a measurable space and $N\in\N_+$ be a positive natural number. We consider a generic random field $Z$ which is indexed by $\Z^N$, i.e., a collection of random variables $\{Z_{s}: {s}\in \Z^N\}$ where each $Z_s$ takes values in $T$. 
$Z$ is (strictly) stationary if for each $k\in \N_+$, for all points $s_1,\ldots,s_k \in \Z^N$ and for each translation $w\in \Z^N$, the joint distribution of the translated vector $( Z_{s_1+w},\ldots,Z_{s_k+w} )$ is equal to the joint distribution of $(  Z_{s_1},\ldots,Z_{s_k}  )$.

Denote the Euclidean maximum norm by $\norm{\cdot}_{\max}$ and define for two subsets $I, J\subseteq \Z^N$ their distance by $ d_{\infty}(I,J) = \inf\{ \norm{s-t}_{\max}: s\in I, t\in J \}$. Furthermore, we write ${s}\le {t}$ if and only if $s_i \le t_i$ for each $i=1,\ldots,N$.

Set ${e_N} = (1,\ldots,1)\in \Z^N$. Let ${n}=(n_1,\ldots,n_N)\in \N^N$, then we write $I_{n}$ for the $N$-dimensional cube on the lattice which is spanned by ${e_N}$ and ${n}$, i.e., $I_{n} = \{ {s}\in \Z^N: {e_N}\le {s} \le {n} \}$. Consider a sequence $(n(k): k\in\N)\subseteq \N^N$ such that
$$ \liminf_{k\rightarrow\infty } \; { \min(n_i(k):i=1,\ldots,N) }/{\max(n_i(k):i=1,\ldots,N)} > 0 $$ and $\lim_{k\rightarrow \infty} n_{i,k} = \infty$ for all $i=1,\ldots,N$. We say that such a sequence converges to infinity and write $n\rightarrow \infty$. Moreover, if $(A_{n(k)}:k\in\N)$ is sequence which is indexed by the sequence $(n_k:k\in\N)$, we also write $A_n$ for this sequence. In particular, we characterize limits for real-valued sequences $A_n$ in this notation, i.e., we agree to write $\lim_{n\rightarrow\infty} A_n$ for $\lim_{k\rightarrow \infty} A_{n(k)}$. $\limsup$ and $\liminf$ are to be understood in the analogue way.
Furthermore, we write $\norm{U}_{\p,p}$ for the $p$-norm of a real-valued random variable $U\in\pspace$.

The $\alpha$-mixing coefficient describes the dependence between random variables, it was introduced by \cite{rosenblatt1956central} and is defined for two sub-$\sigma$-algebras $\cF, \cG$ of $\cA$ by
$
		\alpha(\cF,\cG) \coloneqq \sup \left\{  \left| \p(A\cap B)-\p(A)\p(B)		\right|: A\in\cF, B\in\cG \right\}.
$
Denote by $\cF(I) \coloneqq \sigma( Z_{s}: {s}\in I )$ the $\sigma$-algebra generated by the $Z_{s}$ for $s\in I$ where $I\subseteq\Z^N$. The $\alpha$-mixing coefficient of the random field $Z$ is then defined as
\begin{align}\label{StrongSpatialMixing}
	\alpha(k) \coloneqq \sup_{ I, J \subseteq \Z^N,\; d_{\infty}(I,J)\ge k } \alpha( \cF(I),\cF(J)), \quad k\in\N.
\end{align}
The random field $Z$ is said to be strongly (spatial) mixing (or $\alpha$-mixing) if $\alpha(k)\rightarrow 0$ as $k\rightarrow \infty$. 

In general the strong mixing condition can fail even for Markov processes if certain smoothness conditions are not satisfied. For instance, consider the stationary AR(1) process $X_k = 1/2 (X_{k-1} + \epsilon_k)$ where the innovations are Bernoulli distributed. This process fails to be strongly mixing see \cite{andrews1984non}. In particular, $(X_k: k\in \N)$ does not satisfy any mixing condition which is stricter than $\alpha$-mixing.

Thus, beside the $\alpha$-mixing condition, we shall study processes which satisfy a weak dependence criterion, introduced in \cite{ dedecker2005new}.
Consider the class of (nonlinear) operators mapping from a measurable space $(\cS,\fS)$ to the real numbers. Define for such an operator the supremum norm by $\norm{g}_{\infty} = \sup_{x\in\cS} |g(x)|$ and write
\begin{align}\label{DefinitionCC}
	\cC=\{ g: \cS\rightarrow\R, \norm{g}_{\infty}<\infty\}.
	\end{align}
Moreover, let $\norm{\cdot}^{\sim}$ be a pseudo-norm on $\mathcal{C}$ (which is intended to measure the roughness of an element of $\cC$). For example, a possible choice is the pseudo-norm associated with Lipschitz- or the H{\"o}lder-constant of the operator. Another choice could be some measure for the total variation of the operator $g$. Write
$
			\cC_1 \coloneqq \{g\in\cC, \norm{g}^\sim \le 1 \}
$
for the bounded operators which have a pseudo-norm of at most 1. We define the $\phi_{\cC}$-dependence coefficient between a random variable $X$ which takes values in $\cS$ and a sub-$\sigma$-algebra $\cM\subseteq \cA$ by
\begin{align}\label{DefPhiC}
		\phi_{\cC} (\cM,X) \coloneqq \sup\{ \norm{\E{g(X)|\cM} - \E{g(X)} }_{\p,\infty} : g\in \cC_1 \}. 
\end{align}		
It follows from this definition in \eqref{DefPhiC} that 
$$
			\phi_{\cC}(\cM,X) = \sup\left\{ |\operatorname{Cov}(Z,g(X))| : Z \text{ is } \cM-\text{measurable,} \norm{Z}_{\p,1} \le 1 \text{ and } g\in \cC_1 \right\},
$$
see \cite{dedecker2005new} Lemma 4.

In the following, we shall study the stationary spatial process $(X_s,y_s)$ where the $X_s$ take values in the space $\cS$ and the $y_s$ are real-valued and bounded by a constant $B\coloneqq \norm{y_s}_{\p,\infty}<\infty$. In this case, we define the following variant of \eqref{DefPhiC} which corresponds to the approach of \cite{maume2006exponential} for finite-dimensional time series: consider the $\sigma$-algebra $\cM_k \coloneqq \sigma\{ (X_s,y_s): 1\le \norm{s}_{\max} \le k \}$ and define for $i\in\N$
\begin{align}\label{DefPhiCV}
		\phi_{\cC,y_s}(i) \coloneqq \sup\left\{ \norm{ \E{\frac{y_s}{B} g(X_s) \Big| \cM_k} - \E{\frac{y_s}{B} g(X_s)} }_{\p,\infty}, g\in \cC_1, s \in \N^N, \norm{s}_{\max} = k+i \right\}.
\end{align}
We say that the process $\{(X_s,y_s):s\in\Z^N\}$ is $\cC$-weakly dependent if the coefficients $\phi_{\cC,y_s}(i)$ are summable. If we only consider the univariate process $\{X_s:s\in\Z^N\}$, we formally replace the $y_s$ by ones in the above definition and write $\phi_{\cC}$ instead of $\phi_{\cC,1}$. If the coefficients $\phi_{\cC}(i)$ are summable, we say that $\{X_s: s\in\Z^N\}$ is $\cC$-weakly dependent.

Consider a time series $\{X_t:t\in\Z\}$ and a $\sigma$-algebra $\cM_k$ generated by the time series up to some time $k$. Let $i\in\N_+$ and assume that the time series is $\cC$-weakly dependent. Interpreting the definition of $\phi_{\cC}$ from \eqref{DefPhiCV}, we see that $\phi_{\cC}(\cM, X_{t+k})$ considers only a finite time in the future which is one main difference of a $\cC$-weakly dependent process when compared to ($\alpha$-)mixing processes.

\section{Exponential inequalities for \texorpdfstring{$\alpha$}{alpha}-mixing processes on \texorpdfstring{$N$}{N}-dimensional lattices}\label{Section_ExponentialInequalities}

We begin with an exponential inequality for strongly mixing real-valued random fields. The proofs do not only rely on the concept of splitting the index set in big blocks and small blocks, we additionally exploit the idea of \cite{merlevede2009} who give exponential inequalities for $\alpha$-mixing time series. The key idea is that the sum of a discrete time series on $\{1,\ldots,T\}$ can be understood as an integral of a piecewise constant process on the interval $(0,T]$; this interval is then partitioned in Cantor set-like elements. We generalize this concept to a spatial index set $I_n$.

\begin{proposition}\label{BernsteinLatticeImproved}
Let the real-valued random field $Z$ have exponentially decreasing $\alpha$-mixing coefficients, i.e., there are $c_0, c_1 \in \R_+$ such that the coefficient from \eqref{StrongSpatialMixing} satisfies $\alpha(k ) \le c_0 \exp( - c_1 k)$. The $Z_s$ have expectation zero and are bounded by $B$. Let ${n}\in \N^N$ be such that
\begin{align} \min\{ n_i: i=1,\ldots,N \} \ge C' \max\{ n_i: i=1,\ldots,N \} \label{EqRatioIndex}
\end{align}
for a constant $C'> 0$ and $\min\{n_i : i=1,\ldots,N\} \ge 2^{N+1}$. Define
$ \tilde{C} \coloneqq 2^{-N} \wedge c_1C'^{N/(N+1)} 2^{-(N+1)}$.
Moreover, let $\beta>0$ such that
$$
 \beta B \le \left\{ \tilde{C} / |I_{n}|^{N/(N+1)} \vee 1/|I_{n}| \right\} \vee \left\{ \left(	C' \tilde{C}^{(N+1)/N^2} /2^{N+3} \right)^{ N^2/(N+1)} \wedge \frac{c_1 C'}{2^{N+2}} \Big/ |I_{n}|^{(N-1)/N}   \right\}.
$$ 
Then there are constants $A_1, A_2\in \R_+$ which depend on the lattice dimension $N$, the constant $C'$ and the bound on the mixing coefficients but not on $n\in\N^N$ and not on $B$ such that 
\begin{align}\begin{split}\label{BernsteinLatticeImprovedEq0}
		\log \E{ \exp\left\{ \beta \sum_{s \in I_{n}} Z_s  \right\} } &\le 	A_1 (\beta B)^2 |I_{n}| \left(1+|I_{n}|^{(N-1)/N}  \log |I_{n}| \right) + A_1 \beta B |I_{n}| \exp\left( - A_2  (\beta B)^{-1/N} \right) \\
		&\quad +  A_1 (\beta B )^{(N+1)/N} |I_{n}| \exp\left\{ - A_2 (\beta B)^{1-(N+1)/N^2} |I_{n}|^{(N-1)/N} \right\}.
\end{split}\end{align}
\end{proposition}
\begin{proof}
Throughout the proof we use the convention to abbreviate constants by $C$. Define $\floor{s} \coloneqq (\floor{s_1},\ldots,\floor{s_N})$ for $s\in\R^d$. We extend the process $Z$ to the entire $\R^N$ with the definition $Z_s \coloneqq Z_{\floor{s}}$. In the same way, we extend the definition of the mixing coefficients consistently, $\alpha(z) = \alpha( \floor{z})$ for $z\in\R_+$. We have $\sum_{s\in I_{n}} Z_s = \int_{({e_N},n+{e_N}]} Z_s\intd{s}$, this corresponds to $\int_{(0,n]} Z_s\intd{s}$ for the process which is translated by $-{e_N}$. Write $\textbf{A} \coloneqq \prod_{i=1}^N A_i$ for the volume of the cube $(0,A]$ and set $\underline{A} \coloneqq \min\{A_k:k\in\N\}$. The proof is divided in part (A) and part (B).

We begin with part (A).
Consider the Laplace transform $\E{ \exp\left(\beta \int_{(0,A]} Z_s\intd{s} \right) }$ for $A\in\R^N$ such that $A$ satisfies \eqref{EqRatioIndex}. Firstly, we show that for a suitable constant $C^*$
\begin{align}\label{BernsteinLatticeImprovedEq1}
		\E{ \exp\left( \beta \int_{(0,A]} Z_s\intd{s} \right) } \le  \exp( C^* 2^{2N} \beta^2 B^2 \textbf{A} ) + c_0 \textbf{A}^{1/(N+1)} \exp\left( - \frac{c_1}{2} \underline{A}^{N/(N+1)} \right)
\end{align}
if 
$$\beta B \le \left[\frac{1}{2^N \textbf{A}^{N/(N+1)}} \wedge  \frac{c_1 C'^{N/(N+1)} }{2^{N+1} \textbf{A}^{N/(N+1)} }  \right] \vee 1/\textbf{A} \text{ and } \underline{A} \ge 2^{N+1}.$$
The proof is divided in two steps. In the first step, let $\beta B \textbf{A} \le 1$. We use that $e^x \le 1 + x+x^2$ for $x\le 1$ to deduce
\begin{align}
		\E{ \exp\left\{ \beta \int_{(0,A]} Z_s \intd{s} \right\} }&\le \exp\left\{ \E{ \left( \beta \int_{(0,A]} Z_s \intd{s}  \right)^2 } \right\} \label{EqDavydov0} \\
		&\le \exp\left\{	\beta^2 \int_{(0,A]} \int_{(0,A]} \E{ Z_s Z_t } \intd{s} \intd{t} \right \}  \label{EqDavydov}
\end{align}				
We can bound this last inequality \eqref{EqDavydov} with a result of \cite{davydov1968convergence} and obtain the upper bound
\begin{align*}				
				 \exp\left\{ \beta^2	\int_{(0,A]} \int_{(0,A]} \alpha( \norm{s-t}_{\max} ) B^2 \intd{s} \intd{t} \right \} \le \exp( C^* \beta^2 B^2 \textbf{A} )
\end{align*}
for a $C^*= \eta \int_{0}^{\infty} \alpha(u) u^{N-1} \intd{u}$ where $\eta$ is a constant which depends on the lattice dimension $N$. This implies \eqref{BernsteinLatticeImprovedEq1} and finishes the first step.

In the second step, let $\beta B \textbf{A} > 1$. Set $P_k = A_k^{N/(N+1)}$ and split each coordinate of the cube $(0,A]$ into intervals of length $2P_k$. $P_k$ needs not to be an integer (for $k=1,\ldots,N$). Set $U \coloneqq \prod_{k=1}^N \ceil{ A_k/(2P_k)}$. So in each dimension we can cover the interval $(0,A_k]$ by at most $2 \ceil{A_k/(2P_k)}$ disjoint intervals of length $P_k$. More precisely, we define for each $k=1,\ldots,N$ the collection of disjoint intervals
\begin{align*}
		J_{k,1} &= \bigcup_{v=1}^{\ceil{A_k/(2P_k)}} B_{k,v}^{(1)} = \bigcup_{v=1}^{\ceil{A_k/(2P_k)}} ( 2(v-1) P_k, 2(v-1) P_k + P_k], \\
		J_{k,2} &=\bigcup_{v=1}^{\ceil{A_k/(2P_k)}} B_{k,v}^{(2)} = \bigcup_{v=1}^{\ceil{A_k/(2P_k)}} ( 2(v-1) P_k + P_k, 2v P_k].
\end{align*}
We obtain
\begin{align*}
		(0,A] &= \bigtimes_{k=1}^N ( J_{k,1} \cup J_{k,2} ) = \bigcup_{a\in \{1,2\}^N } \bigtimes_{k=1}^N J_{k,a_k} = \bigcup_{a\in \{1,2\}^N } \bigtimes_{k=1}^N \bigcup_{v_k = 1}^{ \ceil{A_k/(2P_k)}} B_{k,v_k}^{(a_k)} \\ 
		&= \bigcup_{a\in \{1,2\}^N } \bigcup_{v_1=1}^{ \ceil{A_1/(2P_1)}} \ldots \bigcup_{v_N=1}^{ \ceil{A_N/(2P_N)}} \bigtimes_{k=1}^N B^{(a_k)}_{k,v_k} = \bigcup_{u=1}^{2^N} \bigcup_{j=1}^{U} I(u,j)
\end{align*}
where $I(u,j)$ equals $\bigtimes_{k=1}^N B^{(a_k)}_{k,v_k}$ for a certain $a\in\{1,2\}^N$ and $(v_1,\ldots,v_N)\in \bigtimes_{k=1}^N \{1,\ldots,\ceil{A_k/(2P_k)}\}$ for each $u=1,\ldots,2^N$ and $j=1,\ldots,U$.

Consequently, the $I(u,r)$ are disjoint cubes with edge lengths $P_k$ and each has a volume of $\textbf{P} = \prod_{k=1}^N P_k$. The distance between two cubes $I(u,r)$ and $I(u,r')$ for $r\neq r'$ is at least $\underline{p}\coloneqq\min_{k=1,\ldots,N} P_k$ for each $u=1,\ldots,2^N$. We can partition the integral as follows
$$
 \int_{(0,A]} Z_s \intd{s} = \sum_{u=1}^{2^N} \sum_{j=1}^{U}  \int_{ I(u,j) } Z_s \intd{s}  = \sum_{u=1}^{2^N} T(u),  \text{ where } T(u) = \sum_{j=1}^{U} \int_{ I(u,j) } Z_s \intd{s} .
$$
We use the inequality of arithmetic and geometric means to derive that
\begin{align}\label{BernsteinLatticeImprovedEq2a}
		\E{ \exp\left( \beta \int_{(0,A]} Z_s\intd{s} \right) } \le \frac{1}{2^N} \sum_{u=1}^{2^N} \E{ \exp\left(2^N \beta T(u)	\right) }.
\end{align}
Moreover, we obtain for the Laplace transform of $T(u)$ with the lemma of \cite{ibragimov1962some} (Lemma~\ref{IbragimovAlphaMixing}) the bound
\begin{align}\label{BernsteinLatticeImprovedEq2}
		\E{ \exp\left(2^N \beta T(u) \right) } &\le \prod_{j=1}^{U}	\E{ \exp\left(2^N \beta \int_{ I(u,j) } Z_s\intd{s}		\right) }	+ \alpha( \underline{p} ) \, U \exp\left(2^N \beta B \textbf{P} U		\right).
\end{align}
By assumption, we have $\underline{A} \ge 2^{N+1}$ which entails that $A_k / (2P_k) \ge 1$, thus, $\ceil{A_k/(2P_k)}\le A_k/P_k$ for each $k=1,\ldots,N$ and $U \le \textbf{A}/\textbf{P}$. Furthermore, we have $2^N\beta B \textbf{P} \le 1$, i.e., $\beta B \le 1/\left(2^N \textbf{A}^{N/(N+1)} \right)$. Next, we need the assumption that the mixing coefficients satisfy $\alpha(z) \le c_0\exp(-c_1 z)$ for all $z\in\R_+$. We use the same approximation within each cube $I(u,j)$ as in the above lines starting with Equation~\eqref{EqDavydov0} and obtain
\begin{align}
	\eqref{BernsteinLatticeImprovedEq2} &\le \exp\left( C(\beta B)^2 2^{2N} \textbf{P} U \right)  + c_0 \frac{ \textbf{A}}{\textbf{P}} \exp\left( - c_1 \underline{p} + 2^N \beta B \textbf{P} U \right) \nonumber \\
	&\le \exp( C^* 2^{2N} \beta^2 B^2 \textbf{A} ) + c_0 \textbf{A}^{1/(N+1)} \exp\left(-\frac{c_1}{2} \underline{A}^{N/(N+1)} \right). \label{BernsteinLatticeImprovedEq3}
\end{align}
Here we use for the $\exp$ factor in the second term the requirement that
$$ \beta B \le \frac{c_1 C'^{N/(N+1)} }{2^{N+1} \textbf{A}^{N/(N+1)} }, $$
which implies $c_1/2 \cdot \underline{A}^{N/(N+1)} \ge 2^N \beta B \textbf{A}$.
Set now $\tilde{C}\coloneqq 1/2^N \wedge c_1 C'^{N/(N+1)}/2^{N+1}$. Combining \eqref{BernsteinLatticeImprovedEq2a} with equations \eqref{BernsteinLatticeImprovedEq2} and \eqref{BernsteinLatticeImprovedEq3}, we obtain \eqref{BernsteinLatticeImprovedEq1} provided that both
\begin{align}\label{BernsteinLatticeImprovedEq3b}
	\beta B \le  \tilde{C} / \textbf{A}^{N/(N+1)} \vee 1/\textbf{A} \text{ and } \underline{A} \ge 2^{N+1}.
	\end{align}

In part (B), we assume that 
$$ 
	\tilde{C} / \textbf{A}^{N/(N+1)} \vee 1/\textbf{A} < \beta B \le \left(	C' \tilde{C}^{(N+1)/N^2} /2^{N+3} \right)^{ N^2/(N+1)} \wedge \frac{1}{2}\frac{c_1 C'}{2^{N+1}} \frac{1}{\textbf{A}^{(N-1)/N} } .
	$$
We follow the ideas of \cite{merlevede2009} and partition the cube $(0,A]$ in Cantor set-like elements. Therefore, let $\delta \in (0,1)$ be defined as follows
\begin{align}\label{BernsteinLatticeImprovedEq3c}
		\delta \coloneqq \frac{2^{N+1}}{c_1} \beta B \frac{ \textbf{A}}{\underline{A}}.
		\end{align}
By assumption, we have that $\underline{A}\ge C' A_k$ for $k=1,\ldots,N$ and that $\beta B \le \frac{1}{2}\frac{c_1 C'}{2^{N+1}} \textbf{A}^{(1-N)/N}$, thus $\delta \le 1/2$.

We partition each interval $(0,A_k]$ into a middle interval of length $\delta A_k$ and two outer intervals each of length $(1-\delta)/2 A_k$. The outer intervals form outer cubes within the cube $(0,A]$ of measure $(1-\delta)^N/2^N \textbf{A}$, there are $2^N$ outer cubes in total. The remaining number $3^N-2^N$ form those cubes which have at least in one dimension $k$ an edge length of $\delta A_k$ and for which one edge is an inner interval. The total measure of the outer cubes is $2^N \cdot (1-\delta)^N/2^N \textbf{A} = (1-\delta)^N \textbf{A}$, the measure of the residual cubes is $(1-(1-\delta)^N)\textbf{A}$. Denote by $\{O^{(1)}_j : j=1,\ldots,2^N \}$ the collection of the outer cubes. Then the Laplace transform can be bounded as
\begin{align}
		 \E{ \exp\left( \beta \int_{(0,A]} Z_s \intd{s} \right) } &\le \E{ \exp\left( \beta \int_{ \bigcup_{j=1}^{2^N} O^{(1)}_j } Z_s \intd{s} \right) } \exp\left\{	\beta B \textbf{A} (1-(1-\delta)^N )	\right\} \nonumber \\
		\begin{split}\label{BernsteinLatticeImprovedEq4}
		&\le \left\{ \prod_{j=1}^{2^N} \E{\exp\left( \beta \int_{ O^{(1)}_j}  Z_s \intd{s} \right) } + \alpha( \delta \underline{A} ) 2^N  \prod_{j=1}^{2^N} \exp\left(	\beta B \textbf{A} \left(\frac{1-\delta}{2}\right)^N \right) \right\} \\
		&\quad \cdot \exp\left\{	\beta B \textbf{A} (1-(1-\delta)^N )	\right\},
		\end{split}
\end{align}
where the last Equation~\eqref{BernsteinLatticeImprovedEq4} is once more a result of \cite{ibragimov1962some}. Next, use the relation $|\log x - \log y| \le |x-y|$ if $x,y \ge 1$ to obtain for the logarithm of the Laplace transform with the help of \eqref{BernsteinLatticeImprovedEq4} the upper bound
\begin{align}
		&\log \E{ \exp\left( \beta \int_{(0,A]} Z_s \intd{s} \right) } \nonumber \\
		\begin{split}
		&\le \sum_{j=1}^{2^N} \log \E{\exp\left( \beta \int_{ O^{(1)}_j}  Z_s \intd{s} \right) } \\
		&\quad + 2^N c_0 \exp\left( -c_1 \underline{A}\delta + 2^N \beta B \textbf{A} \left(\frac{1-\delta}{2} \right)^N \right) + \beta B \textbf{A} \left(1-(1-\delta)^N \right).\label{BernsteinLatticeImprovedEq5}
		\end{split}
\end{align}
We can repeat the computations for the Laplace transform on the sets $O^{(1)}_{j_1}$. By formally replacing the cube $(0,A]$ with the cube $O^{(1)}_{j_1}$, we obtain a similar bound in terms of new outer subcubes w.r.t.\ $O^{(1)}_{j_1}$, these are given by
$$ \left\{ O^{(2)}_{j_2}: j_2 = 1+(j_1-1)2^N,\ldots,2^N+(j_1-1)2^N \right\}$$
for $j_1 = 1,\ldots,2^N$. Here we have to replace in \eqref{BernsteinLatticeImprovedEq5} as well $\underline{A}$ by $\underline{A}\frac{1-\delta}{2}$ and $\textbf{A}$ by $\textbf{A} \left(\frac{1-\delta}{2}\right)^N$ . Next, define the number $l$ by
$$ l \coloneqq \inf\left\{ k\in \Z: (\beta B)^{(N+1)/N} \textbf{A} \left(\frac{1-\delta}{2}\right)^{Nk} \le \tilde{C}^{(N+1)/N} \right\}, $$
where $\tilde{C} = 1/2^N \wedge c_1 C'^{N/(N+1)} /2^{N+1} $. Note that this definition is meaningful because we are in the case where $(\beta B)^{(N+1)/N} \textbf{A} > \tilde{C}^{(N+1)/N}$. Write $O^{(0)}_1$ for the cube $(0,A]$.

After further $l-1$ iterations of \eqref{BernsteinLatticeImprovedEq5}, we obtain the following bound with the sets $\left\{ O^{(l)}_{j_l}: j_l= 1,\ldots, 2^{Nl} \right\}$
\begin{align}
		\log \E{ \exp\left( \beta \int_{(0,A]} Z_s \intd{s} \right) } &= \log \E{ \exp\left( \beta \int_{O^{(0)}_1 } Z_s \intd{s} \right) } \nonumber \\
		\begin{split}\label{BernsteinLatticeImprovedEq6}
		&= \sum_{j_l=1}^{2^{Nl}} \log \E{ \exp\left( \beta \int_{ O^{(l)}_{j_l} } Z_s \intd{s} \right) } \\
		&\quad + \sum_{j=0}^{l-1} \beta B \textbf{A} \left( \frac{1-\delta}{2} \right)^{Nj} \left(1-(1-\delta)^N \right) 2^{Nj} \\
		&\quad + \sum_{j=0}^{l-1} c_0 2^{N(j+1)} \exp\left\{	- c_1 \underline{A} \left(\frac{1-\delta}{2}\right)^j \delta+ 2^N \beta B \textbf{A} \left(\frac{1-\delta}{2}\right)^{N(j+1 )} \right\}.
		\end{split}
\end{align}
We can bound the three sums in \eqref{BernsteinLatticeImprovedEq6}, therefore we use the following inequalities which follow from the definition of $l$
$$ 2^{Nl} \le C (\beta B)^{(N+1)/N} \textbf{A}, \quad l \le C \log \textbf{A}  \text{ and }  \textbf{A} \left(\frac{1-\delta}{2}\right)^{N(l-1)} > \left(\frac{\tilde{C}}{\beta B}\right)^{(N+1)/N}. $$
The second sum in \eqref{BernsteinLatticeImprovedEq6} is at most 
\begin{align}\label{BernsteinLatticeImprovedEq7}
		\sum_{j=0}^{l-1} \beta B \textbf{A} \left( \frac{1-\delta}{2} \right)^{Nj} \left(1-(1-\delta)^N \right) 2^{Nj} &\le \beta B \textbf{A} (1-(1-\delta)^N ) l \le \beta \delta B \textbf{A} l \le C \beta \delta B \textbf{A} \log \textbf{A}.
\end{align}
Next, we apply the inequality from \eqref{BernsteinLatticeImprovedEq1} to bound the first sum in \eqref{BernsteinLatticeImprovedEq6}. Therefore, we need that the requirements of \eqref{BernsteinLatticeImprovedEq3b} are satisfied: it follows from the definition of $l$ that $ \beta B \le \tilde{C}  \big/ \left( \textbf{A} \left(\frac{1-\delta}{2}\right)^{Nl} \right)^{N/(N+1)}$. Moreover, we need that $\underline{A} \left( \frac{1-\delta}{2} \right)^l \ge 2^{N+1}$: using the fact that $\delta \le 1/2$, we find
\begin{align*}
		\underline{A} \left(\frac{1-\delta}{2}\right)^l \ge C' \left(	\textbf{A}\left(\frac{1-\delta}{2}\right)^{N(l-1)}	\right)^{1/N} \frac{1-\delta}{2} \ge C' \left( \frac{\tilde{C}}{\beta B} \right)^{(N+1)/N^2} \frac{1}{4} \ge 2^{N+1}.
\end{align*}
The last inequality follows because $\beta B \le \left(	C' \tilde{C}^{(N+1)/N^2} /2^{N+3} \right)^{ N^2/(N+1)}$. Hence, the first sum in \eqref{BernsteinLatticeImprovedEq6} can be estimated similarly as in \eqref{BernsteinLatticeImprovedEq2}:
\begin{align}\label{BernsteinLatticeImprovedEq8}
	&2^{Nl} \log  \E{ \exp\left( \beta \int_{ O^{(l)}_{1} } Z_s \intd{s} \right) } \nonumber \\
	&\le 2^{Nl} \left\{ C^* 2^{2N} \beta^2 B^2 \textbf{A} \left( \frac{1-\delta}{2}\right)^{Nl} + c_0 \textbf{A}^{1/(N+1)} \left( \frac{1-\delta}{2}\right)^{Nl/(N+1)} \exp\left( - \frac{c_1}{2} \underline{A}^{N/(N+1)} \left( \frac{1-\delta}{2}\right)^{Nl/(N+1)} \right) \right\} \nonumber \\
	&= C^* 2^{2N} (\beta B)^2 \textbf{A} (1-\delta)^{Nl} + c_0 \textbf{A}^{1/(N+1)} (1-\delta)^{Nl/(N+1)} 2^{N^2 l/(N+1)} \exp\left( - \frac{c_1}{2} \underline{A}^{N/(N+1)} \left( \frac{1-\delta}{2}\right)^{Nl/(N+1)} \right) \nonumber \\
	&\le C (\beta B)^2 \textbf{A} + C  \beta B \textbf{A} \exp\left( - \frac{c_1}{2} C'^{N/(N+1)}  \left( \frac{1-\delta}{2}\right)^{N/(N+1)} \left( \frac{\tilde{C}}{\beta B} \right)^{1/N} \right).
\end{align}
Consequently using the definition of $\delta$, we can bound \eqref{BernsteinLatticeImprovedEq7} and \eqref{BernsteinLatticeImprovedEq8} together by
\begin{align}\label{BernsteinLatticeImprovedEq9}
	C (\beta B)^2 \textbf{A} \left(1+\textbf{A}^{1-1/N}  \log \textbf{A} \right) + C \beta B \textbf{A} \exp\left( - C  (\beta B)^{-1/N} \right).
\end{align}
For the third sum in \eqref{BernsteinLatticeImprovedEq6} we need the condition that $\frac{c_1}{2} \underline{A}\delta \ge \beta B \textbf{A} (1-\delta)^N$. This is implied by the definition of $\delta$ from \eqref{BernsteinLatticeImprovedEq3c}, thus,
\begin{align}\label{BernsteinLatticeImprovedEq10}
	&\sum_{j=0}^{l-1} c_0 2^{N(j+1)} \exp\left\{	- c_1 \underline{A}\delta \left(\frac{1-\delta}{2}\right)^j + 2^N \beta B \textbf{A} \left(\frac{1-\delta}{2}\right)^{N(j+1 )} \right\} \nonumber \\
	&\le \sum_{j=0}^{l-1} c_0 2^{N(j+1)} \exp\left\{	- \frac{c_1}{2} \underline{A}\delta \left(\frac{1-\delta}{2}\right)^j \right\} \nonumber \\
	&\le 2^N c_0 \frac{2^{Nl}-1}{2^N-1} \exp\left\{ -\frac{c_1}{2} \delta \underline{A} \left(\frac{1-\delta}{2} \right)^l \right\} \nonumber \\
	&\le C (\beta B)^{(N+1)/N} \textbf{A} \exp\left\{ -\frac{c_1}{2} \delta \underline{A} \left(\frac{1-\delta}{2} \right)^l \right\} \nonumber \\
	& \le C (\beta B )^{(N+1)/N} \textbf{A} \exp\left\{ - C (\beta B)^{1-(N+1)/N^2} \textbf{A}^{1-1/N} \right\}.
\end{align}
Hence, combining \eqref{BernsteinLatticeImprovedEq9} with \eqref{BernsteinLatticeImprovedEq10} yields
\begin{align}\begin{split}\label{BernsteinLatticeImprovedEq11}
		\log \E{ \exp\left\{ \beta \int_{(0,A]} Z_s\intd{s} \right\} } &\le 	C (\beta B)^2 \textbf{A} \left(1+\textbf{A}^{1-1/N}  \log \textbf{A} \right) + C \beta B \textbf{A} \exp\left( - C  (\beta B)^{-1/N} \right) \\
		&\quad +  C (\beta B )^{(N+1)/N} \textbf{A} \exp\left\{ - C (\beta B)^{1-(N+1)/N^2} \textbf{A}^{1-1/N} \right\}
		\end{split}
\end{align}
if $\tilde{C} / \textbf{A}^{N/(N+1)} <\beta B \le c_1 C'/2^{N+2} \textbf{A}^{-(N-1)/N} \wedge \left(	C' \tilde{C}^{(N+1)/N^2} /2^{N+3} \right)^{ N^2/(N+1)}$.
Comparing \eqref{BernsteinLatticeImprovedEq11} with \eqref{BernsteinLatticeImprovedEq1} in the case that $\beta B \le \tilde{C} /\textbf{A}^{N/(N+1)}$ yields the result.
\end{proof}

The proof of Proposition~\ref{BernsteinLatticeImproved} reveals that for spatial data the rate of convergence is determined by the fact that the distance between the blocks decays at a rate $\underline{p}$, however, the number of observations within a block is at least $\underline{p}^N$. Compare the last term (resp. factor) on the right-hand side of \eqref{BernsteinLatticeImprovedEq2} (resp. \eqref{BernsteinLatticeImprovedEq4}) which in both cases is due to the $\alpha$-mixing property, see the lemma of \cite{ibragimov1962some}. So if $N>1$, the decreasing mixing coefficient can not fully compensate for the sample which grows like a polynomial of degree $N$. We see this in the next corollary which shows that the exponential decay is determined by the effective sample size $|I_n|^{1/N}$.

\begin{corollary}\label{CorBernsteinImproved}
Let the real-valued random field $Z$ satisfy all conditions from Proposition~\ref{BernsteinLatticeImproved}. Then there are constants $A_1,A_2\in\R_+$ such that for all $\epsilon > 0$
\begin{align*}
		\p\Big( |I_{n}|^{-1} \Big| \sum_{s \in I_{n}} Z_s \Big| \ge \epsilon \Big) \le A_1 \exp\left( - A_2 \frac{\epsilon}{B} \frac{ |I_{n}|^{1/N} }{ (\log |I_{n}|)^2 } \right).
\end{align*}
\end{corollary}
\begin{proof}
Choose $\beta \propto (B |I_n|^{(N-1)/N} (\log |I_n|)^2 )^{-1}$. Then we infer from Proposition~\ref{BernsteinLatticeImproved} that this choice is admissible (if $n$ is sufficiently large). Furthermore, we obtain with Markov's inequality
\begin{align}\label{CorBernsteinImprovedEq0}
		\p\left( |I_n|^{-1}\left| \sum_{s \in I_{n}} Z_s \right| \ge \epsilon \right) \le 2 \exp\left(-\beta |I_n| \epsilon \right) \E{\exp\left( \beta \sum_{s\in I_n} Z_s \right) }.
\end{align}
Thus, the expression inside the first $\exp$-factor is proportional to $ \beta |I_n| \propto |I_n|^{1/N} / B (\log |I_n|)^2$. Furthermore, a comparison with the requirements of Proposition~\ref{BernsteinLatticeImproved} shows that it remains to compute the quantities
\begin{align*}
		& (\beta B)^2 |I_n| |I_n|^{(N-1)/N} \log |I_n| \propto |I_n|^{1/N} / (\log |I_n|)^{3} \text{ and } (\beta B)^{(N+1)/N} |I_n| \propto |I_n|^{1/N^2} / (\log |I_n| )^{2(N+1)/N}.
\end{align*}
Hence, the first $\exp$-factor in \eqref{CorBernsteinImprovedEq0} dominates the second $\exp$-factor and we obtain the desired result.
\end{proof}

Next, we give an exponential inequality for centered Hilbert space-valued random variables. Therefore, we need two conditions: the first states that the tail of the entire distribution vanishes at an exponential rate. The second requires that the contribution of a further marginal dimension decays exponentially as well. In particular, this last assumption is not uncommon, see e.g., \cite{bosq_linear_2000}.

\begin{theorem}\label{BernsteinHilbert}
Let $\cH$ be a separable Hilbert space with inner product $\scalar{\cdot}{\cdot}$ and orthonormal basis $\{ e_j: j\in \N \}$. Let $\{ Z_{s}: {s}\in \Z^N \}$ be a random field on $\Z^N$, $N\in\N_+$, the marginals of which take values in $\cH$ and satisfy $\E{Z_{s} } =0$. The $Z_{s}$ satisfy uniformly in $s\in\Z^N$ the conditions
\begin{align}\label{EqBernsteinHilbert0} 
 \E{ \scalar{Z_{s}}{e_j}^2 } \le d_0 \exp(-d_1 j ) \text{ for all } j\in\N \text{ and }  \p( \hnorm{Z_{s}} \ge z ) \le \kappa_0 \exp( -\kappa_1 z^{\gamma} ) 
\end{align}
for positive constants $d_0,d_1\kappa_0,\kappa_1$ and $\gamma$. The mixing coefficients of the random field decrease exponentially as in Proposition~\ref{BernsteinLatticeImproved} and there is a lower bound $C'$ for the ratio between the smallest and the largest coordinate of $n$ as in Equation \eqref{EqRatioIndex}. Moreover, let $\epsilon>0$. Then there are constants $A_1$ and $A_2$ such that
\begin{align*}
		\p\left( |I_{n}|^{-1}	 \hnorm{\sum_{s\in I_{n}} Z_{s} } \ge \epsilon \right) &\le A_1 \exp\left\{  - A_2 \left(\frac{ \epsilon |I_{n}|^{1/N} }{(\log |I_{n}|)^2}\right)^{2\gamma/ (2+3\gamma)} \right\} .
\end{align*}
$A_1$ and $A_2$ depend on the decay rate of the mixing coefficients, on the tail parameters $\gamma,\kappa_i,d_i$ and on $C'$ but not on ${n}$. Additionally, $A_1$ depends polynomially on $|I_n|$ and $\epsilon$. If additionally $\gamma \ge 1$, $A_1$ does not depend on $\epsilon$ and $|I_n|$ and
\begin{align*}
		\p\left( |I_{n}|^{-1}	  \hnorm{\sum_{{s}\in I_{n}} Z_{s} } \ge \epsilon \right) &\le A_1 \exp\left\{  - A_2 \left(\frac{ \epsilon |I_{n}|^{1/N} }{(\log |I_{n}|)^2}\right)^{2/5 } \right\} \\
		&\quad\cdot \Biggl\{	\epsilon^{-2} + \left(\frac{ \epsilon |I_{n}|^{1/N} }{(\log |I_{n}|)^2}\right)^{2/5} +\left(\frac{|I_{n}|^{1/N} }{(\log |I_{n}|)^2}\right)^{1/5 }\; \epsilon^{- 4/5} \Biggl\}. 
\end{align*}
\end{theorem}
\begin{proof}
Following \cite{bosq_linear_2000}, we decompose the sum $S_{n} = \sum_{{s}\in I_{n}} Z_{s}$ in a finite-dimensional part and a remainder. Then we bound the latter with the help of the decay in the single coordinates and apply the exponential inequality for finite-dimensional random variables to the first part. More precisely, the following decomposition is true for each natural number $m$
\begin{align}
		\p\left( |I_{n}|^{-1} \hnorm{S_{n}} \ge \epsilon \right) &\le 	\p\left( \sum_{j=1}^m \scalar{S_{n}}{e_j}^2 \ge (|I_{n}| \epsilon /2 )^2 \right)  + \p\left( \sum_{j=m+1}^{\infty} \scalar{S_{n}}{e_j}^2 \ge (|I_{n}| \epsilon /2 )^2  \right) \nonumber \\
		&\le  \sum_{j=1}^m  \p\left(\scalar{S_{n}}{e_j}^2 \ge \frac{ (|I_{n}| \epsilon )^2}{4 m}  \right) + \E{  \sum_{j=m+1}^{\infty} \scalar{S_{n}}{e_j}^2  } \left(  \frac{ \epsilon |I_n|}{2 } \right)^{-2} \nonumber \\
		&\le m\cdot \max_{1\le j \le m} \p\left( |\scalar{S_{n}}{e_j} | \ge \frac{ |I_{n}| \epsilon}{2 \sqrt{m} } \right) + \left( \frac{2}{\epsilon  } \right)^2 \sum_{j=m+1}^{\infty}  \E{ \scalar{Z_{e_N}}{e_j}^2 } . \label{EqBernsteinHilbert1}
\end{align}
By assumption, there are $d_0,d_1\in\R_+$ such that $\sum_{j=m+1}^{\infty}  \E{ \scalar{Z_{e_N}}{e_j}^2 } \le \sum_{j=m+1}^{\infty} d_0 \exp(-d_1 j )$. Hence, the second term in \eqref{EqBernsteinHilbert1} decays at an exponential rate. Note that we do not use a covariance inequality for $\alpha$-mixing spatial processes for the second term in \eqref{EqBernsteinHilbert1} at this point because it would not increase significantly the overall rate of convergence. We apply the inequality from Proposition~\ref{BernsteinLatticeImproved} to the first term and use the assumption that the tail of the random variables decays exponentially, i.e., $\p( |Z_{s}| \ge z ) \le \kappa_0 \exp( -\kappa_1 z^{\gamma} ) $. We obtain with similar arguments as in \cite{valenzuela2017bernstein}
\begin{align}
		\p\left( |\scalar{S_{n}}{e_j} | \ge  |I_{n}| \epsilon \right) \le \inf_{ D > 0} \left( A_1 D^{1-\gamma} \epsilon ^{-1} \exp\left\{ - A_2 D^{\gamma} \right\} + A_1 \exp\left\{ -A_2 \frac{\epsilon  |I_{n}|^{1/N} }{D (\log |I_{n}|)^2} \right\} \right), \label{EqBernsteinHilbert2}
\end{align}
where the constants $A_1$ and $A_2$ only depend on the coefficients $\kappa_0$ and $\kappa_1$ which bound the tail of the distribution, the lattice dimension $N$ and the mixing coefficients. We can approximately equate both terms in \eqref{EqBernsteinHilbert2} with the choice $D = \left( R({n})\epsilon \right)^{1/(1+\gamma)}$, where $R({n}) \coloneqq |I_{n}|^{1/N} / (\log |I_{n}|)^2$. In particular, we obtain for \eqref{EqBernsteinHilbert1} the following asymptotic bound if we insert \eqref{EqBernsteinHilbert2} for the finite-dimensional part
\begin{align}\begin{split}\label{EqBernsteinHilbert3}
		\p\left( |I_{n}|^{-1}	 \hnorm{S_{n}} \ge \epsilon \right) &\le \inf_{m\in\N} \Biggl(  A_1 m \left\{ 1+ R({n})^{ (1-\gamma)/(1+\gamma)} \left(\frac{\epsilon}{\sqrt{m}}\right) ^{-2\gamma /(1+\gamma) } \right\} \\
		&\qquad\qquad\qquad \cdot \exp\left\{ -A_2 \left( \frac{ R({n})\epsilon}{\sqrt{m}} \right)^{\gamma/(1+\gamma)} \right\} + A_3 \frac{ \exp\left\{ -A_4 m		\right\} }{\epsilon^2  } \Biggl).
\end{split}\end{align}
Here the constants $A_1,\ldots,A_4$ do not depend on $m$ and $n$. Again, both term are approximately equal for the choice $m \coloneqq \floor{ ( R({n}) \epsilon)^{2\gamma /(2+3\gamma)}	}$. In this case, \eqref{EqBernsteinHilbert3} reduces to
\begin{align*}
		\p\left( |I_{n}|^{-1}	 \hnorm{S_{n}} \ge \epsilon \right) &\le A_1 \exp\left\{  - A_2 \left( \epsilon R({n}) \right)^{2\gamma/ (2+3\gamma)} \right\} \Biggl\{		 \epsilon^{-2} + \left( R({n})\epsilon \right)^{2\gamma / (2+3\gamma)} \\
		&\quad+ R({n})^{2(1+\gamma-\gamma^2)/[ (2+3\gamma)(1+\gamma) ] }\; \epsilon^{- \gamma(3+5\gamma) /[ (2+3\gamma)(1+\gamma)] } \Biggl\}. 
\end{align*}
This finishes the proof.
\end{proof}

\section{Exponential inequalities for \texorpdfstring{$\cC$}{C}-weakly dependent spatial processes}\label{Section_ExpInequalitiesPhiMixing}
The aim of this section is to derive exponential inequalities for $\cC$-weakly dependent spatial processes. We assume for the next proposition that $\{ (X_s,y_s):s\in\N^N\}$ is a stationary random field. The $X_s$ take values in the Banach space $\cS$, the $y_s$ are real-valued and bounded by $B\coloneqq \norm{y_s}_{\p,\infty} <\infty$. $\norm{\cdot}^\sim$ is a pseudo-norm on the space of operators $\cC$ from Equation~\eqref{DefinitionCC}. Note that $\cC$ contains elements which are not necessarily linear and that the coefficients $\phi_{\cC,y_s}(i)$ from \eqref{DefPhiCV} depend on the choice of the pseudo-norm. We obtain with these assumptions:

\begin{proposition}\label{WeaklyCDependent}
Let $\{ (X_s,y_s):s\in\N^N\}$ be stationary such that the coefficients from \eqref{DefPhiCV} satisfy $\sum_{i=1}^{\infty} \phi_{\cC,y_s}(i)<\infty$. Let $\tilde{g}: \cS\rightarrow \R$ be a bounded operator, i.e., $\sup_{x\in\cS} |\tilde{g}(x)|<\infty$. Define
 $S_n =  \sum_{s\in I_n} y_s \tilde{g}(X_s)$.
Then there are constants $A_1,A_2$ which depend on the lattice dimension $N$, the coefficients $\phi_{\cC,y_s}$ but neither on $n\in\N^N$ nor on $B$ such that
\begin{align}\label{ExpPhiMixing}
		\p\left( |I_n|^{-1} \left| S_n - \E{ S_n} \right| \ge \epsilon  \right) \le A_1 \exp\left( -A_2 \epsilon ^2 |I_n|^{1/N}  (\norm{\tilde{g}}^{\sim})^{-1} B^{-2} \right).
\end{align}
\end{proposition}
\begin{proof}
We write $\norm{\cdot}$ for the maximum norm on $\N^N$ and partition the sum $\sum_{s\in I_n} y_s \tilde{g}(X_s) $ as follows: we collect all indices with equal maximum norm and set 
$$
	Z_k = \sum_{\substack{ s\in I_n,\\ \norm{s}=k }} y_s \tilde{g}(X_s). \text{ Then } \sum_{s\in I_n} y_s \tilde{g}(X_s) = \sum_{k=1}^{\norm{n}} \sum_{\substack{ s\in I_n,\\ \norm{s}=k }} y_s \tilde{g}(X_s) = \sum_{k=1}^{\norm{n}} Z_k. 
$$
Denote by $\tilde{\cM}_k$ the $\sigma$-algebra generated by $\{Z_0,\ldots,Z_k\}$. We derive from Proposition 4 of \cite{dedecker2003new} that
\begin{align}\label{ExpIneqEq1}
		\norm{ \sum_{k=1}^{\norm{n}} Z_k -\E{Z_k}}_{\p,p} \le \left( 2p \sum_{k=1}^{\norm{n}} b_{k,\norm{n}} \right)^{1/2},
\end{align}
for $p\ge 2$ and where the coefficients $b_{k,\norm{n}}$ equal
\begin{align*}
		b_{k, \norm{n} } = \max_{k\le l\le \norm{n}} \norm{ (Z_k - \E{Z_k} ) \sum_{i=k}^l  \E{Z_i \big| \tilde{M}_k } - \E{Z_i}  }_{\p,p/2}.
\end{align*}
Note that $\norm{Z_k}_{\infty} = \cO(B k^{N-1})$. Hence, the coefficients $b_{k,\norm{n}}$ satisfy the inequality
\begin{align*}
		b_{k,\norm{n}} &\le \norm{ Z_k - \E{Z_k}  }_{\infty} \sum_{i=k}^{\norm{n}} \norm{ \E{Z_i \big| \tilde{M}_k } - \E{Z_i} }_{\infty} \le C B^2 k^{N-1} \sum_{i=k}^{\norm{n}} i^{N-1} \phi_{\cC,y_s}(i-k) \norm{\tilde{g}}^{\sim} ,
\end{align*}
where $\phi_{\cC,y_s}(i)$ is defined in \eqref{DefPhiCV}. Thus, \eqref{ExpIneqEq1} is at most (modulo a constant which depends on the lattice dimension $N$)    
\begin{align}
		&\left(2p \sum_{k=1}^{\norm{n}} B^2 k^{N-1} \sum_{i=k}^{\norm{n}} i^{N-1}\phi_{\cC,y_s}(i-k) \norm{\tilde{g}}^{\sim}   \right)^{1/2} =\left( 2p\norm{\tilde{g}}^{\sim} B^2 \norm{n}^{N-1} \sum_{i=0}^{\norm{n}-1} \phi_{\cC,y_s}(i) \sum_{k=1}^{\norm{n}-i}  (i+k)^{N-1} \right)^{1/2} \nonumber \\
		&\le C \left( 2p\norm{\tilde{g}}^{\sim} B^2 \norm{n}^{N-1} \sum_{i=0}^{\norm{n}-1} \phi_{\cC,y_s}(i)  \left( (\norm{n}+1)^{N} - (i+1)^N \right) \right)^{1/2}. \label{ExpIneqEq2}
\end{align}
Following Proposition 5 in \cite{dedecker2005new}, we obtain from this $L^p$-inequality the desired exponential inequality which is given in Equation~\eqref{ExpPhiMixing}: we obtain with Markov's inequality
\begin{align*}
		\p\left( |I_n|^{-1} \left| S_n - \E{ S_n} \right| \ge \epsilon  \right) &\le 
1 \wedge \inf_{p\ge 2} (\epsilon  |I_n|)^{-p} \, \E{  \left| S_n - \E{ S_n} \right|^p } \\
&\le 1 \wedge \inf_{p\ge 2} C_1 \, (\epsilon  |I_n|)^{-p} \, \left( C_2\, p\norm{\tilde{g}}^{\sim} B^2 \norm{n}^{2N-1 } \sum_{i=0}^{\infty} \phi_{\cC,y_s}(i)  \right)^{p/2} \\
&\le 1 \wedge \inf_{p\ge 2} C_3 \left( C_4\,  \epsilon ^{-2}   p\norm{\tilde{g}}^{\sim} B^2 |I_n|^{-1/N }   \right)^{p/2}
\end{align*}
for certain constants $C_1,\ldots,C_4$. Now, as demonstrated in \cite{dedecker2005new} this is bounded by the $\exp$-expression in Equation~\eqref{ExpPhiMixing}.
\end{proof}

The analogue of Theorem~\ref{BernsteinHilbert} for $\cC$-weakly dependent data is given in terms of a stationary random field $\{(X_s,Y_s): s\in\N^N \}$ where the $X_s$ are $\cS$-valued and the $Y_s$ are $\cH$-valued. Again, $\cH$ is a separable Hilbert space which is equipped with an orthonormal basis $\{e_j:j\in\N\}$.

\begin{theorem}\label{WeaklyCDependentHilbert}
Assume that the tail of the distribution of the $Y_s$ admits the exponential bounds as in Equation~\eqref{EqBernsteinHilbert0}.
Set $y_{j,s} \coloneqq \scalar{Y_s}{e_j}$ and $y_{j,s}^{(B)} \coloneqq \min(B,\max(-B,y_{j,s}))$ for $B>0$.  Moreover, assume that 
\begin{align}\label{UniBddPhi}
		\sup_{j\in\N} \sup_{B>0} \sum_{i\in\N} \phi_{\cC,y^{(B)}_{j,s}}(i) < \infty,
\end{align}
where the $\phi_{\cC,y^{(B)}_{j,s}}$ are defined in \eqref{DefPhiCV}. Let $\tilde{g}\in\cC_1$ and set $S_n =  \sum_{s\in I_n} Y_s \tilde{g}(X_s) \in \cH$. Then
\begin{align}
		&\p\left( |I_n|^{-1} \hnorm{ S_n - \E{ S_n } } \ge \epsilon \right) \nonumber\\
		\begin{split}\label{WeaklyCDepEq0}
		&\le A_1 \left[ \epsilon^{-2} + m + m^{(4+5\gamma)/(4+2\gamma)} \epsilon^{-3\gamma/(2+\gamma)} \left(|I_n|^{1/N} \left(\norm{\tilde{g}}^{\sim}\right)^{-1} \right)^{(1-\gamma)/(2+\gamma)} \right] \\
		&\quad\cdot \exp\left\{ - A_2 \left( \epsilon^2 |I_n|^{1/N}   \left(\norm{\tilde{g}}^{\sim}\right)^{-1} \right)^{\gamma/(2+2\gamma)}  \right\},
\end{split}\end{align}
where $m= \left(\epsilon^2 |I_n|^{1/N}  \left(\norm{\tilde{g}}^{\sim}\right)^{-1} \right)^{\gamma/(2+2\gamma)}$. In particular, if $\gamma\ge 1$,
\begin{align*}
		\p\left( |I_n|^{-1} \hnorm{ S_n - \E{ S_n } } \ge \epsilon \right) &\le \left[\epsilon^{-2}+\left(\epsilon^2 |I_n|^{1/N} \left(\norm{\tilde{g}}^{\sim}\right)^{-1} \right)^{1/4} + \left(\epsilon^2 |I_n|^{1/N} \left(\norm{\tilde{g}}^{\sim}\right)^{-1} \right)^{9/24} \epsilon ^{-1} \right] \\
		&\quad \cdot \exp\left\{ - A_2 \left( \epsilon^2 |I_n|^{1/N}   \left(\norm{\tilde{g}}^{\sim}\right)^{-1} \right)^{1/4}  \right\}.
\end{align*}
\end{theorem}
\begin{proof}
We proceed as in the proof of Theorem~\ref{BernsteinHilbert} and use the result from Proposition~\ref{WeaklyCDependent}. After splitting the sum in a finite-dimensional part and an infinite-dimensional remainder, we end up in a constellation as in Equation~\eqref{EqBernsteinHilbert1}:
\begin{align*}
\p\left( |I_{n}|^{-1}  \hnorm{S_{n}- \E{ S_n}} \ge \epsilon \right) 	&\le m\cdot \max_{1\le j \le m} \p\left( |\scalar{S_{n}}{e_j} | \ge \frac{ |I_{n}|  \epsilon}{2 \sqrt{m} } \right) + \left( \frac{2}{\epsilon } \right)^2 \sum_{j=m+1}^{\infty}\E{\scalar{Z_{e_N}}{e_j}}^2.
		\end{align*}
The finite-dimensional part needs to be split in a part bounded by a constant $B$ as well as a positive and negative remainder. More precisely, we write $y_{j,s} = y_{j,s}^{(B)} + \max( y_{j,s}-B,0) + \min( y_{j,s}+B,0)$. Hence, if we use additionally the fact that the tail of the distribution of the $y_{j,s}$ is uniformly bounded, we obtain for the finite-dimensional part (similar as in Equation \eqref{EqBernsteinHilbert2} and using Proposition~\ref{WeaklyCDependent}) the bound
\begin{align*}
		\p\left(	|\scalar{S_{n}}{e_j} |\ge \frac{ \epsilon  |I_n| }{2\sqrt{m}}	\right) &\le A_1 \inf_{B>0} \Bigg\{ B^{1-\gamma} \epsilon^{-1} m^{1/2} \exp(-A_2 B^{\gamma})  + \exp\left(-A_2 \epsilon^2 |I_n|^{1/N} \left(\norm{\tilde{g}}^{\sim}\right)^{-1} B^{-2} m^{-1} \right) \Bigg\}.
\end{align*}
Note that the uniform boundedness of the weak dependence coefficients from Equation~\eqref{UniBddPhi} is necessary in order to apply Proposition~\ref{WeaklyCDependent} uniformly in $j$. Consequently, the choice $B=\left( \epsilon^2 |I_n|^{1/N} \left(\norm{\tilde{g}}^{\sim}\right)^{-1} m^{-1} \right)^{1/(2+\gamma)}$ yields 
\begin{align*}
		\p\left( |I_{n}|^{-1} \hnorm{S_{n}- \E{ S_n}} \ge \epsilon \right) 
		&\le A_1 \inf_{m\in\N} \Biggl\{	m \left[1+
		\left( |I_n|^{1/N} \left(\norm{\tilde{g}}^{\sim}\right)^{-1} \right)^{(1-\gamma)/(2+\gamma)} m^{3\gamma/(2(2+\gamma))} \epsilon ^{-3\gamma/(2+\gamma)} 	
			\right]\\
		&\quad \cdot \exp\left[	- A_2 \left( \epsilon^2 |I_n|^{1/N}  \left(\norm{\tilde{g}}^{\sim}\right)^{-1} \right)^{\gamma/(2+\gamma)} m^{-\gamma/(2+\gamma)} \right] + \left(\frac{2}{\epsilon}\right)^2\exp( -A_2 m ) \Biggl \}.
\end{align*}
Choosing $m$ proportional to $(\epsilon^2 |I_n|^{1/N} \left(\norm{\tilde{g}}^{\sim}\right)^{-1})^{\gamma/(2+2\gamma)}$ yields the rate in \eqref{WeaklyCDepEq0}.
\end{proof}

\section{Applications in the functional kernel regression model}\label{Section_Application}

In this section, let $\mathcal{D}$ be a convex and compact subset of $\R^d$. The Hilbert space $\cH$ is given by the function space $L^2( \cD, \cB(\cD), \nu)$ over the field $\R$, where $\nu$ is a finite measure, e.g., the Lebesgue measure or a probability measure. The inner product on $\cH$ is $\scalar{x}{y} = \int_\mathcal{D} xy \intd{\nu}$. We assume that $\cS$ is a superset of the continuous functions on $\cD$ and a subset of $\cH$, i.e., $C^0(\cD) \subseteq \cS \subseteq \cH$.

Consider a pseudo-metric $d$ on $\cS$ which satisfies $d(x,y) \le \hnorm{x-y}= (\int_\cD |x-y|^2\intd{\nu})^{1/2}$ for all $x,y \in \cS$. An example for $d$ would be a projection-based pseudo-metric. We study the strictly stationary process $( (X_s,Y_s) : s\in\Z^N)$, $N\in\N_+$, where $X_s$ takes in $\cS$ and $Y_s$ takes values in $\cH$. The process satisfies the functional regression model
\begin{align}\label{GenX}
		Y_s = \Psi(X_s) + \epsilon_s,\quad s\in\Z^N
		\end{align}
where the error terms $\epsilon_s$ are $\cH$-valued with $\E{ \epsilon_s | X_s } =0$.% and finite second moment $\E{ \hnorm{\epsilon_s}^2 } < \infty$. %Similarly, we assume that $\E{\hnorm{\Psi(X_s)}^2<\infty}$.

We estimate the operator $\Psi: \cS \rightarrow\cH $ with the methods from the kernel regression framework of \cite{ferraty2004nonparametric}, \cite{ferraty_nonparametric_2007} and \cite{ferraty2012regression}. An important variable in this model is the small ball probability function which is defined with the help of $d$ as $F_x(h) = \p( d(X_s,x)\le h )$, for $h\ge0$. Let $K$ be a kernel function; we write $K_h \coloneqq K( \cdot /h)$ and estimate the operator $\Psi$ pointwise by
\begin{align}\begin{split}\label{DefHatPsi}
		&\hat{\Psi}_h(x) \coloneqq \frac{ \hat{g}_h(x) }{ \hat{f}_h(x) } \in\cH, \quad \text{ for } x\in \cS, \text{ where }\\
		&\qquad\qquad \hat{f}_h (x) \coloneqq (|I_n| F_x(h))^{-1} \sum_{s\in I_n} K_h (d(X_s,x) ) \in \R \text{ and }\\
		 &\qquad\qquad\qquad\qquad \hat{g}_h (x)\coloneqq (|I_n| F_x(h))^{-1} \sum_{s\in I_n} Y_s K_h ( d(X_s,x) ) \in \cH. 
\end{split}\end{align}

$\cH$ is equipped with an orthonormal basis $\{e_j:j\in\N\}$. Denote by $\psi_{j} \coloneqq \scalar{\Psi(\cdot)}{e_j}$ the $j$-th coordinate of the operator $\Psi$ w.r.t.\ the orthonormal basis and by $y_{j,s} \coloneqq \scalar{Y_s}{e_j}$ the $j$-th coordinate of the process $Y_s$. Set $y_{j,s}^{(B)} \coloneqq \min( B, \max(-B,y_{j,s}))$ for $B\ge0$. Moreover, define $\vartheta_{x,j}(s) \coloneqq \E{ \psi_j(X_s)-\psi_j(x) | d(X_s,x)=s }$ for $j\in\N$ and $x\in\cS$. We write $\norm{x}_{\infty}$ for the essential supremum of a function $x$ on $\mathcal{D}$ w.r.t.\ $\nu$ and make the following assumptions:

\begin{enumerate}

\item \label{Condition1} $\Psi\colon \cS \to\cH$ is uniformly H{\"o}lder continuous of order $r$ w.r.t.\ $\hnorm{\cdot}$, i.e., $\hnorm{\Psi(x)-\Psi(y)}\le L_\Psi \hnorm{x-y}^r$ for some $r\in(0,1]$.
For some $\delta>0$, all $0\le u \le \delta$, all $j\in\N$ and all $x\in\cS$, $\vartheta_{x,j}(0)=0$, $\vartheta'_{x,j}(u)$ exists and $\vartheta'_{x,j}(u)$ is uniformly H{\"o}lder continuous of order $r$, i.e., there is a $0<L_{x,j}< \infty$ such that $|\vartheta'_{x,j}(u)-\vartheta'_{x,j}(0)|\le L_{x,j} u^r$ for all $0\le u\le \delta$. Additionally, $\sup_{x\in\cS} \sum_{j\in\N} \vartheta'_{x,j}(0)^2 < \infty$ and $\sup_{x\in\cS} \sum_{j\in\N} L_{x,j}^2 <\infty$.

\item \label{Condition2} the kernel $K$ has support in $[0,1]$ and has a continuous derivative $K'\le 0$. The Lipschitz constant of $K$ on [0,1] is denoted by $L_K$, i.e., $\left| K(u)-K(v)\right| \le L_K |u-v|$ for all $u,v\in [0,1]$.

\item \label{Condition3} K(1) = 0, which implies that the kernel function is Lipschitz continuous on $\R_+$.

\item \label{Condition4}
	the small ball probability $F_x(h)=\p(d(X_s,x)\le h)$ is positive for all $h>0$ and for all $x\in \cS$. The limit of the quotient $\tau_x(u)\coloneqq \lim_{h\downarrow 0 }  F_x(hu) /F_x(h)$
exists for all $u\in[0,1]$ and all $x\in\cS$ and it is uniform:
\begin{align}\label{EqSmallBall}
		\lim_{h\downarrow 0} \;\sup_{x\in\cS}\; \sup_{u\in[0,1]} \left| \frac{ F_x(hu)}{F_x(h)} - \tau_x(u) \right| = 0.
\end{align}

\item \label{Condition5} $M_x \coloneqq K(1)-\int_0^1 K'(u)\tau_x(u)\,\intd{u} > 0$ for all $x\in\cS$ and $\inf_{x\in\cS}M_x >0$.

\item\label{Condition6}  there is a $\delta>0$ such that the small ball probability quotient 
\begin{align*}
		\cS \times[0,1] \ni (z,u) \mapsto  F_z(hu)/ F_x(h)
\end{align*}
is Lipschitz continuous for each fixed point $x\in\cS$ with Lipschitz constant $L_{x} $ which is uniform in $h$ for $h\le \delta$.

\item \label{Condition7} the tail of the distribution of the $Y_s$ decays exponentially, i.e., $\p( \hnorm{Y_s} \ge z ) \le \kappa_0 \exp(-\kappa_1 z^{\gamma})$ for some $\gamma \ge 1$. Furthermore, there are positive constants $d_0,d_1$ such that
$$ \E{ \scalar{ Y_s}{e_j}^2 } \le d_0 \exp(-d_1 j). $$

\item \label{Condition8} set $\tilde{\vartheta}_x(u) = \E{ \norm{Y_s} | d(X_s,x)=u }$. Then $\sup_{x\in\cS, \norm{x}_{\infty}\le R} \tilde{\vartheta}_x(0) = \cO(R^r)$. Moreover, there is a $\delta>0$ such that for all $x\in\cS$ and $0\le u\le \delta$ the derivative $\tilde{\vartheta}'_x(u)$ exists and $\sup_{x\in\cS, u\le \delta} |\tilde{\vartheta}'_x(u)| < \infty$.

\item \label{Condition9} the process $\{ (X_s,Y_s): s\in \N^N \}$ is strongly spatial mixing with exponentially decreasing mixing coefficients such that $\alpha(k) \le c_0 \exp( -c_1 k)$ for $\alpha$ defined as in Equation~\eqref{StrongSpatialMixing}.

\item \label{Condition10} the pseudo-norm on $\cC$ from \eqref{DefinitionCC} is defined by
$
			\norm{g}^{\sim} \coloneqq \sup_{x,y\in \cS, x\neq y}  |g(x)-g(y)| / \hnorm{x-y}.
$
The process $(X,Y)$ is uniformly $\cC$-weakly dependent in the sense that the coordinate processes of the $Y_s$ satisfy
$$
		\sup_{j\in\N} \sup_{B>0} \sum_{i\in\N} \phi_{\cC,y^{(B)}_{j,s}}(i) < \infty \text{ and } \sum_{i\in \N } \phi_{\cC}(i) < \infty.
$$

\end{enumerate}

Condition \ref{Condition1} ensures that the regression operator is uniformly continuous on $\cS\subseteq \cH$, w.r.t.\ the norm $\hnorm{\cdot}$ which is stronger than the pseudo-metric $d$. The requirement on the conditional expectation functions is not uncommon, a similar assumption is made in \cite{ferraty2012regression}. It ensures in particular that the conditional expectation of the difference of the full operator $\Psi(X_s)-\Psi(x)$ admits a meaningful first order expansion w.r.t.\ $d(X_s,x)$. Condition~\ref{Condition2} contains standard assumptions on the kernel, see \cite{ferraty_nonparametric_2007}. For the concept of weak dependence, we need in the following that the kernel function $K$ is continuous, thus, in this case Condition~\ref{Condition3} is additionally necessary.

Condition~\ref{Condition4} can be motivated by the following observation: since the underlying Hilbert space is a function space, one has in many applications that for a point $x$ in the Hilbert space $\p( \norm{ X_s - x} \le h ) \sim C(x) \p( \norm{X_s} \le h)$ for $h \downarrow 0$. For further details see e.g. \cite{ferraty2006estimating}, \cite{ferraty_nonparametric_2007} and \cite{ferraty2012regression}. %\cite{ferraty2004nonparametric} investigate the convergence of the estimator $\hat{\Psi}$ on a compact set $\cK$ in the case where the response variable is real-valued. They assume that the small ball probability function is uniformly bounded as follows: there are two constants $b_1, b_2$ and a function $\rho$ such that
%$$ \liminf_{h\downarrow 0} \int_0^1 \rho(ht)\intd{t}/\rho(h) > 0 \text{ and } 0 < b_1 \rho(h) \le F_x(h) \le b_2 \rho(h)  .$$
%for all $x\in\cK$.

The positivity of the moments $M_x$ in Condition~\ref{Condition5} is technical and guaranteed if $K(1)>0$. In the same way, Conditions~\ref{Condition6} to ~\ref{Condition8} guarantee certain technical properties of the estimator $\hat{\Psi}$ in the subsequent proofs. Condition~\ref{Condition9} is not unusual if we assume that the data are $\alpha$-mixing and is also mentioned in \cite{ferraty2004nonparametric}. In the same way, Condition~\ref{Condition10} guarantees a solution if the data are $\cC$-weakly dependent.

Define on $C^0(\cD)$ the norm
\begin{align}\label{NormC0}
 \norm{x}_{1,C^0(\cD)} \coloneqq \sup_{u\in \mathcal{D} } |x(u)| + \sup_{u,v\in \mathcal{D}, u\neq v } \frac{ \left| x(u) - x(v) \right|}{ \norm{u-v} }.
\end{align}
Consider for $R>0$ the $\delta$-covering number $ N(\cG(R),\delta,\hnorm{\cdot})$ of the set $	\cG(R) \coloneqq \left \{ x\in C^0(\mathcal{D}): \norm{x}_{1,C^0(\cD)} \le R \right\}$ w.r.t.\ the norm $\hnorm{\cdot}$. Then the following is well known:
\begin{lemma}\label{TotallyBoundednessFunctionalCase}
The set $\cG(R)$ is totally bounded and there is a constant $C$ which only depends on $d$ such that the covering number w.r.t.\ the $\hnorm{\cdot}$-norm on the function space $\cH$ satisfies 
$ \log N(\cG(R),\delta,\hnorm{\cdot}) \le C \lambda( \mathcal{D} ^1 ) ( \sqrt{\nu(\mathcal{D})}R  / \delta  )^{d}$, 
where $\lambda(\cD^1) = \{u\in\R^d: \; \exists v\in\cD:\, \norm{u-v}_{\infty}\le 1\}$.
\end{lemma}
\begin{proof}
By Theorem 2.7.1 in \cite{van2013weak} the logarithm of the covering number of $\cG(1)$ w.r.t.\ the supremum norm can be bounded by $\lambda( \mathcal{D} ^1 ) (1 / \delta )^d$ times a constant which only depends on $d$. Now, note that the covering number of $\cG(R)$ w.r.t.\ the 2-norm on $\mathcal{D}$ can be bounded by the $\delta/\sqrt{\nu(\mathcal{D})}$-covering number of $\cG(R)$ w.r.t.\ the supremum norm on $\mathcal{D}$ which in turn can be bounded by the $\delta/(R\sqrt{\nu(\mathcal{D})})$-covering number of $\cG(1)$ w.r.t.\ the supremum norm on $\mathcal{D}$. This finishes the proof.
\end{proof}

\begin{lemma}\label{UnifConvHatF}
$\lim_{h\rightarrow 0} \sup\{ | \E{ \Kh{d(X_0,x)} F_x(h)^{-1} } - M_x |: x\in \cS \} = 0$. In particular, $\E{\hat{f}_h(x) } \rightarrow M_x$ uniformly in $x\in\cS$ for any choice of the bandwidth $h=h_n$ which vanishes if $n$ converges to infinity.
\end{lemma}
\begin{proof}
The claim follows from the assumption of the uniform convergence of the small ball probability and the expansion provided in \cite{ferraty_nonparametric_2007}. Let $x\in\cS$ be fixed, then
\begin{align*}
		\left| \E{ \frac{\Kh{d(X_0,x)}}{F_x(h) } }-M_x \right| &= \left| \int_0^1 K'(u) \left( \frac{ F_x(hu)}{F_x(h)} - \tau_x(u) \right) \intd{u} \right| \\
		&\le \int_0^1 |K'(u)|\intd{u}\; \sup_{x\in\cS} \sup_{u\in[0,1]} \left|  \frac{ F_x(hu)}{F_x(h)} - \tau_x(u)  \right| \rightarrow 0.
\end{align*}
The last inequality is independent of $x\in\cS$.
\end{proof}

We give two results on the consistency of the estimator $\hat{\Psi}$. The first one applies to the case where the data is strongly spatial mixing, the second one applies to $\cC$-weakly dependent data.

For both results the number $\inf_{x\in\cG(R)} F_x(h)$ will be of interest. It depends on the bandwidth $h$, the radius $R$ of the set $\cG(R)$ and on the spatial process $X$ itself. So $R$, $\inf_{x\in\cG(R)} F_x(h)$ and $h$ can be mutually dependent in a complex way which is of particular interest if $R$ converges to infinity. This has also consequences for the proofs of the upcoming Theorem~\ref{UniformConvHatPsi} and Theorem~\ref{UniformConvHatPsiV} where we need to construct a $\delta$-covering of the set of functions $\cG(R)$ which depends on the radius $R$. To avoid this dependence, we choose $\delta$ only to depend on the sample size $|I_n|$ and not on the numbers $R$, $\inf_{x\in\cG(R)} F_x(h)$ and $h$.

\begin{theorem}[Uniform convergence under strong spatial mixing conditions]\label{UniformConvHatPsi}
Let Conditions (\ref{Condition1}), (\ref{Condition2}) and (\ref{Condition4}) - (\ref{Condition9}) be satisfied. Let $(n_k:k\in\N)$ be a sequence in $\N^N$ which converges to infinity. Let $R_n$ be a real-valued sequence which has a limit in $(0,\infty]$ and assume that the bandwidth $h=h_n$ converges to zero such that
 $$
\frac{ R_n^{5d/2} (\log |I_n|)^7 }{  |I_n|^{1/N\cdot 2/(5d+2)} \, \inf_{x\in\cG(R_n) } F_{x} (h) } \rightarrow 0 \text{ and } \frac{R_n^r}{ |I_n|^{1/N \cdot 2/(5d+2) } \, h }  \rightarrow 0.
$$
Then 
$$
		\sup_{x\in \cG(R_n) } \hnorm{ \hat{\Psi}_h(x) - \Psi(x) }  = \cO\left( \frac{ R_n^{5d/2} (\log |I_n|)^{7} }{ |I_n|^{1/N \cdot 2/(5d+2) } \, \inf_{x\in\cG(R_n) } F_x(h) } \right) + \cO\left(\frac{R_n^r}{|I_n|^{1/N \cdot 2/(5d+2) } \,h  } \right) + \cO\left(		h^r \right) \quad a.s.
$$
\end{theorem}

\begin{proof}[Proof of Theorem \ref{UniformConvHatPsi}]
Before we begin with the proof, we define  $\delta_n \coloneqq |I_n|^{-1/N\cdot 2/(2+5d)}$ and choose a function $V(n)$ which is proportional to 
$$  \frac{ R_n/\delta_n)^{5d/2} (\log |I_n|)^7 }{  \inf_{x\in\cG(R_n) } F_x(h)\, |I_n|^{1/N} } 
$$
and which we will use later. We follow \cite{collomb1977estimation} and consider the difference $\hat{\Psi}_h(x) - \Psi(x)$ on the ball $\cG = \cG(R)$:
\begin{align*}
\hat{\Psi}_h(x) - \Psi(x)  &= (\hat{f}_h(x))^{-1} \Bigg\{	\left( \hat{g}_h(x) - \E{\hat{g}_h(x) } \right) - \Psi(x)\left( \hat{f}_h(x) - \E{\hat{f}_h(x) } \right) \\
&\qquad	\qquad	\qquad	+\left( \E{ \hat{g}_h(x) } - \Psi(x)\E{\hat{f}_h(x)} \right) \Bigg\} .
	\end{align*}
	Thus,
	\begin{align}\begin{split}\label{UniformConvHatPsiEq1}
	\sup_{x\in\cG} \hnorm{ \hat{\Psi}_h(x) - \Psi(x) }
&\le \Bigg\{	\sup_{x\in\cG} \hnorm{ \hat{g}_h(x) - \E{\hat{g}_h(x) } }  +  \sup_{x\in\cG} \hnorm{\Psi(x)} \cdot \sup_{x\in\cG}  \left| \hat{f}_h(x) - \E{\hat{f}_h(x) }  \right|  \\
	&\qquad\qquad + \sup_{x\in\cG}  \hnorm{ \E{ \hat{g}_h(x) } - \Psi(x)\E{\hat{f}_h(x)} } \Bigg\}  \Biggl/ \inf_{x\in\cG} \hat{f}_h(x).
\end{split}\end{align}
The third term in the numerator of \eqref{UniformConvHatPsiEq1} can be bounded by $  \sup_{x\in\cS}  \hnorm{ \E{ \hat{g}_h(x) } - \Psi(x)\E{\hat{f}_h(x)} } = \cO(h^r) $:
\begin{align*}
		\hnorm{ \E{\hat{g}_h(x)-\Psi(x)\hat{f}_h(x) } }^2 &= \sum_{j\in\N}  \E{ (\psi_j(X_s)-\psi_j(x)) \frac{ \Kh{d(X_s,x)} }{F_x(h)} } ^2 = \sum_{j\in\N}  \E{ \vartheta_{x,j}(d(X_s,x)) \frac{ \Kh{d(X_s,x)} }{F_x( h)} } ^2 \\
		&\le 2 \sum_{j\in\N} \E{ \frac{ \Kh{d(X_s,x)} }{F_x( h)} \vartheta'_{x,j}(0) d(X_s,x) }^2 + 2 \sum_{j\in\N} \E{ \frac{ \Kh{d(X_s,x)} }{F_x( h)} L_{x,j} h^r }^2.
\end{align*}
Note that the left-hand side of the last inequality is in $\cO( h^{2r})$ uniformly in $x\in\cS$ because both $\sup_{x\in\cS}\sum_{j\in\N} L_{x,j}^2 <\infty$ and $\sup_{x\in\cS}\sum_{j\in\N} \vartheta'_{x,j}(0)^2<\infty$ and because $\E{ \Kh{d(X_s,x)} / F_x( h)} $ converges uniformly to $M_x$ by Lemma~\ref{UnifConvHatF} and $\sup_{x\in\cS} M_x < \infty$.

The denominator in \eqref{UniformConvHatPsiEq1} can be bounded as
\begin{align}
		\inf_{x\in\cG} \hat{f}_h(x) &\ge \inf_{x\in\cG} \E{ \hat{f}_h(x) } - \sup_{x\in \cG} \left| \hat{f}_h(x) - \E{\hat{f}_h(x)  } \right| \nonumber \\
		&\ge \inf_{x\in\cS} M_x - \sup_{x\in \cS} \left| \E{\hat{f}_h(x)}-M_x \right| - \sup_{x\in\cG} \left| \hat{f}_h(x) - \E{ \hat{f}_h(x)} \right|.\label{UniformConvHatPsiEq2}
\end{align}
By assumption, the infimum on the right-hand side of \eqref{UniformConvHatPsiEq2} is positive and the first supremum converges to zero by Lemma~\ref{UnifConvHatF}. In order to show that the right-hand side of \eqref{UniformConvHatPsiEq2} is positive, it remains to show that the second supremum converges to zero $a.s.$ We demonstrate this implicitly when considering the two remaining terms of the numerator of Equation~\eqref{UniformConvHatPsiEq1}
$$
	\sup_{x\in\cG} \hnorm{ \hat{g}_h(x) - \E{\hat{g}_h(x) } } 
	\text{ and }
	 \sup_{x\in\cG} \hnorm{\Psi(x)}  \cdot \sup_{x\in\cG}  \left| \hat{f}_h(x) - \E{\hat{f}_h(x) }  \right| .
$$
We can bound $\sup_{x\in\cG} \hnorm{\Psi(x)}$ by $\hnorm{\Psi(0)} + L_{\Psi} \nu(\cD)^{r/2} R^r$. In the sequel, we write for simplicity $\norm{\cdot}$ both for $\hnorm{\cdot}$ and $|\cdot|$, so we can treat both cases at the same time. Consider the following generic situation
\begin{align}\label{UniformConvHatPsiEq3}
		\sup_{x\in\cG} \norm{ \frac{1}{|I_n|} \sum_{s\in I_n} \tilde{Y}^{(l)}_s \frac{ \Kh{d(X_s,x)}}{F_x(h)} - \E{\tilde{Y}^{(l)}_s \frac{ \Kh{d(X_s,x)}}{F_x(h)}} },
\end{align}
where $\tilde{Y}^{(l)}_s = Y_{s}$ if $l=1$ and $\tilde{Y}^{(l)}_s = \hnorm{\Psi(0)} + L_{\Psi} \nu(\cD)^{r/2} R^r$ if $l=0$. Next, choose a $\delta_n$-covering of $\cG$ w.r.t.\ the norm $\hnorm{\cdot}$, i.e., there are points $v_1,\ldots,v_m$ such that for all $x\in\cG$ there is a point $v_j$ with the property $d(x,v_j)\le\hnorm{x-v_j}<\delta_n$. The covering number $m\coloneqq N(\cG(R_n),\delta_n,\hnorm{\cdot})$ depends on $\delta_n$. Then we can bound~\eqref{UniformConvHatPsiEq3} as
\begin{align}\label{UniformConvHatPsiEq4}
\begin{split}
		&\max_{1\le j \le m} \norm{ \frac{1}{|I_n|} \sum_{s\in I_n} \tilde{Y}^{(l)}_s \frac{ \Kh{ d(X_s,v_j) }}{F_{v_j} (h)} - \E{\tilde{Y}^{(l)}_s \frac{ \Kh{ d(X_s,v_j) }}{F_{v_j}(h)}} }  \\
		&\quad + \max_{1\le j \le m} \sup_{x\in U_{\delta}(v_j)}  \norm{ \frac{1}{|I_n|} \sum_{s\in I_n} \tilde{Y}^{(l)}_s \left\{ \frac{ \Kh{d(X_s,x)}}{F_{x} (h)} - \frac{ \Kh{ d(X_s,v_j)}}{F_{v_j}(h)} \right\} }  \\
		&\quad  + \max_{1\le j \le m} \sup_{x\in U_{\delta}(v_j)}  \norm{ \E{ \frac{1}{|I_n|} \sum_{s\in I_n} \tilde{Y}^{(l)}_s \left\{ \frac{ \Kh{d(X_s,x)}}{F_{x} (h)} - \frac{ \Kh{ d(X_s,v_j) }}{F_{v_j}(h)} \right\} } }.
		\end{split}
\end{align}
We begin with the first term in \eqref{UniformConvHatPsiEq4} and show that it vanishes $a.s.$ Therefore, we first consider the functional case for the $\tilde{Y}^{(1)}_s=Y_s$. We infer from Theorem~\ref{BernsteinHilbert} and Lemma~\ref{TotallyBoundednessFunctionalCase} that for the choices $\delta=\delta_n$, $R=R_n$ and $h=h_n$ there are generic constants such that
\begin{align}
		&\p\left( \max_{1\le j \le m} \hnorm{ \frac{1}{|I_n|} \sum_{s\in I_n} Y_{s} \frac{ \Kh{d(X_s,v_j)}}{F_{v_j} (h)} - \E{ Y_s \frac{ \Kh{d(X_s,v_j)}}{F_{v_j}(h)}} } 		\ge z\right) \label{UniformConvHatPsiEq4bb} \\
		&\le m\, \max_{1\le j\le m}\, \p\left(  \hnorm{ \frac{1}{|I_n|} \sum_{s\in I_n} Y_s \Kh{d(X_s,v_j)} - \E{ Y_s \Kh{d(X_s,v_j)}} }  		\ge z \inf_{x\in\cG} F_{x} (h) \right) \nonumber \\
		&\le A_1 \exp\left\{ A_2 \left(\frac{R_n}{\delta_n}\right)^d - A_3   \left( \frac{ z \inf_{x\in\cG} F_{x} (h) |I_n|^{1/N} }{(\log |I_n|)^2}\right)^{2/5} \right\} \nonumber \\
		&\qquad\qquad \cdot\left\{ \left( z \inf_{x\in\cG} F_{x} (h) \right)^{-2} +  \left( \frac{ z \inf_{x\in\cG} F_{x} (h) |I_n|^{1/N}}{ (\log |I_n|)^{2}} \right)^{2/5} + \left( z \inf_{x\in\cG} F_{x} (h) \right)^{-4/5} \left(\frac{|I_n|^{1/N}}{(\log |I_n|)^2}\right)^{1/5} \right\} . \nonumber
\end{align}
If we multiply the factor $z$ inside the probability of \eqref{UniformConvHatPsiEq4bb} by $V(n)$, we find that this probability is still summable for a sequence $(n(k):k\in\N)\subseteq \N^N$ which converges to infinity. Hence, it follows from the first Borel-Cantelli Lemma that
\begin{align*}
		&\max_{1\le j \le m} \hnorm{ \frac{1}{|I_n|} \sum_{s\in I_n} Y_{s} \frac{ \Kh{d(X_s,v_j)}}{F_{v_j} (h)} - \E{ Y_s \frac{ \Kh{d(X_s,v_j)}}{F_{v_j}(h)}} } = \cO( V(n)) \quad a.s. \\
		&= \cO\left( \frac{ R_n^{5d/2} (\log |I_n|)^7 }{ \inf_{x\in\cG} F_{x} (h)  |I_n|^{1/N\cdot 2/(5d+2)} } \right) \quad a.s.
\end{align*}
This means in particular that the first summand in \eqref{UniformConvHatPsiEq4} vanishes $a.s.$ in the functional case.

Consider the first term in \eqref{UniformConvHatPsiEq4} in the scalar case $l=0$. Note that $\tilde{Y}^{(0)}_s$ is the same for all $s$. We use the same bound on the covering number as before and obtain with Corollary~\ref{CorBernsteinImproved} generic constants $A_1$, $A_2$ and $A_3$ such that
\begin{align}
		&\p\left(	\max_{1\le j \le m} \left|  \frac{1}{|I_n|} \sum_{s\in I_n} \tilde{Y}^{(0)}_s \frac{ \Kh{d(X_s,v_j)}}{F_{v_j} (h)} - \E{ \tilde{Y}^{(0)}_s \frac{ \Kh{d(X_s,v_j)}}{F_{v_j}(h)}} \right| 		\ge z	\right) \label{UniformConvHatPsiEq4cc} \\
		&\le m\, \max_{1\le j\le m} \p\left(  \left| \frac{1}{|I_n|} \sum_{s\in I_n} \Kh{d(X_s,v_j)} - \E{ \Kh{d(X_s,v_j)}} \right|  		\ge z \inf_{x\in\cG} F_{x} (h) / \tilde{Y}^{(0)}_0 \right) \nonumber \\
		&\le A_1 \exp\left\{ A_2 \left(\frac{R_n}{\delta_n}\right)^d - A_3\, \frac{ z \inf_{x\in\cG} F_{x} (h)  |I_n|^{1/N} ) }{R_n^r ( \log |I_n|)^2 } \right\} .\nonumber 
\end{align}
Arguing similar as before, we infer from Equation~\eqref{UniformConvHatPsiEq4cc} that
\begin{align*}
		&\max_{1\le j \le m} \left|  \frac{1}{|I_n|} \sum_{s\in I_n} \tilde{Y}^{(0)}_s \frac{ \Kh{d(X_s,v_j)}}{F_{v_j} (h)} - \E{ \tilde{Y}^{(0)}_s \frac{ \Kh{d(X_s,v_j)}}{F_{v_j}(h)}} \right| \\
		&= \cO\left(	\frac{ (R_n / \delta_n)^d R_n^r (\log |I_n|)^3 }{ \inf_{x\in\cG} F_{x} (h)  |I_n|^{1/N} } \vee \frac{ R_n^r (\log |I_n|)^4 }{ \inf_{x\in\cG} F_{x} (h)  |I_n|^{1/N} } \right)  \quad a.s.\\
		&= \cO\left( \frac{ R_n^{5d/2} (\log |I_n|)^7 }{ \inf_{x\in\cG} F_{x} (h)  |I_n|^{1/N\cdot 2/(5d+2)} } \right) \quad a.s.
\end{align*}
for a sequence $(n_k:k\in\N)\subseteq\N^N$ which converges to infinity. In particular, the first summand in Equation~\eqref{UniformConvHatPsiEq4} vanishes $a.s.$ in the real case, too. 

Next, we consider the third summand in \eqref{UniformConvHatPsiEq4}, similar considerations apply to the second summand if we use the exponential inequalities from Section~\ref{Section_ExponentialInequalities}, so we do not need to inspect the second summand closer. We use the Lipschitz continuity of the kernel on the interval [0,1] and the uniform Lipschitz continuity of the small ball probability and bound the third summand as
\begin{align}\label{UniformConvHatPsiEq5}
\begin{split}
		&\max_{1\le j \le m} \sup_{x\in U_{\delta}(v_j)}  \E{ \norm{\tilde{Y}^{(l)}_s} \Biggl\{ \left| \frac{ \Kh{d(X_s,x)}}{F_{v_j} (h)} - \frac{ \Kh{d(X_s,v_j)}}{F_{v_j} (h)} \right|\cdot \frac{ F_{v_j}(h)}{F_x(h)} + \frac{\Kh{d(X_s,v_j)}}{ F_{v_j}(h)}\,  \frac{ | F_x(h)-F_{v_j}(h)| }{F_x(h)} \Biggl\} } .
\end{split}
\end{align}
We write $U_{\delta}(y)$ for the $\delta$-neighborhood of $y\in\cS$ w.r.t.\ the metric $d$ throughout the rest of this proof. For the difference in the kernel functions in \eqref{UniformConvHatPsiEq5}, we need to distinguish two cases which are given by the following two inclusions
\begin{align*}
		&\left\{ X_s\in U_h(v_j)\cap U_h(x) \right\} \subseteq \{ X_s\in U_h(v_j) \} \text{ and } \\
		&\qquad\qquad \{X_s\in [U_h(v_j)\setminus U_h(x)] \cup[ U_h(x)\setminus U_h(v_j) ] \} \subseteq \{ X_s \in  U_{h}(v_j)\setminus U_{h-\delta_n}(v_j) \} \cup \{ X_s \in  U_{h}(x)\setminus U_{h-\delta_n}(x) \}.
\end{align*}
Moreover, note that the quotient of the small ball probability functions in Equation~\eqref{UniformConvHatPsiEq5} can be bounded with the help of a fixed reference point in $\cS$, namely 0, as:
$$
		 \frac{ | F_x(h)-F_{v_j}(h)| }{F_x(h)} \le \frac{ F_0(h)}{\inf_{x \in \cG} F_x(h)} L_0 d(x,v_j) \le \frac{ L_0 \delta_n }{\inf_{x \in \cG} F_x(h)}.
$$
Furthermore, we have
$
		F_y(h) / F_x(h) \le 1 + C \delta_n / \inf_{x\in\cG} F_x(h)
$,
whenever $d(x,y)\le \delta_n$, using the Lipschitz continuity of the small ball probability function. Since $\delta_n /\inf_{x\in\cG} F_x(h)$ converges to 0, this implies in particular that the above ratio $F_{v_j}(h)/F_x(h)$ in \eqref{UniformConvHatPsiEq5} is bounded. Thus, we obtain for \eqref{UniformConvHatPsiEq5} modulo a constant the bound 
\begin{align}\label{UniformConvHatPsiEq6}
\begin{split}
&\max_{1\le j \le m}   \mathbb{E} \Biggl[ \norm{\tilde{Y}^{(l)}_s} \frac{\delta_n}{h} \frac{ \1{ X_s \in U_h(v_j)} }{F_{v_j}(h)} \\
&\qquad\qquad\qquad+   \norm{\tilde{Y}^{(l)}_s} \frac{ \1{X_s \in U_{h}(v_j) \setminus U_{h-\delta_n}(v_j) }+ \1{X_s \in U_{h}(x) \setminus U_{h-\delta_n}(x) } }{F_{v_j}(h) } \\
&\qquad\qquad\qquad\qquad\qquad\qquad\qquad\qquad\qquad\qquad  +\norm{\tilde{Y}^{(l)}_s} \frac{ \Kh{ d(X_s,v_j)}}{F_{v_j}(h)} \frac{\delta_n }{\inf_{x\in\cG} F_x(h) } \Biggl].
\end{split}
\end{align}
The first two terms in \eqref{UniformConvHatPsiEq6} are from the difference in the kernel functions, the last one from the difference in the small ball probability functions. We begin with the case $l=0$. Using the uniform convergence result of Lemma~\ref{UnifConvHatF}, we see that the first term in Equation~\eqref{UniformConvHatPsiEq6} is in $\cO( R_n^r \delta_n / h_n) = \cO( R_n^r / (h_n\, |I_n|^{1/N \cdot 2/(5d+2)} ) ) $. 

Similarly, the third term is in $\cO( R_n^r \delta_n / \inf_{x\in\cG} F_x(h)) = \cO( R_n^r /(\inf_{x\in\cG} F_x(h)) |I_n|^{1/N \cdot 2/(5d+2)} ) $. Note that we can bound $R_n^r$ by $R_n^{5d/2} (\log |I_n|)^7$ in the last $\cO$-expression.

For the second term in \eqref{UniformConvHatPsiEq6}, we use the continuity of the quotient of the small ball probability functions w.r.t.\ a fixed reference point to find that this summand is in $\cO( R^r \delta_n / \inf_{x\in\cG} F_x(h))$.

We continue with the case $l=1$ and consider the second term in \eqref{UniformConvHatPsiEq6}. We write $\tilde{\vartheta}_x(u)$ for the conditional expectation function $\E{\norm{Y_s}|d(X_s,x)=u}$ which is assumed to be differentiable in a neighborhood of zero. So we can use a Taylor expansion for the following difference
\begin{align}\begin{split}\label{UniformConvHatPsiEq7}
		\E{ \norm{Y_s} \frac{ \1{X_s \in U_{h}(x) \setminus U_{h-\delta_n}(x) } }{F_{x}(h) } } &= \E{ \left(\tilde{\vartheta}_x(0) + \tilde{\vartheta}'_x(Z_{1,s} ) d(X_s,x) \right)  \frac{ \1{X_s \in U_{h}(x) } }{F_{x}(h) } } \\
		&\quad - \E{ \left(\tilde{\vartheta}_x(0) + \tilde{\vartheta}'_x(Z_{2,s}) d(X_s,x) \right)  \frac{ \1{X_s \in U_{h-\delta_n}(x) } }{F_{x}(h) } } 
\end{split}\end{align}
where the random variables $Z_{1,s}$ and $Z_{2,s}$ are between $x$ and $X_s$. We can give upper bounds on \eqref{UniformConvHatPsiEq7}:
\begin{align*}
		&\tilde{\vartheta}_x(0) \frac{ F_x(h) - F_x(h-\delta_n)}{F_x(h)} + \sup_{u\le h} |\tilde{\vartheta}'_x( u )| h \frac{ F_x(h) + F_x(h-\delta_n)}{F_x(h)} \\
		&\le  C \left( \sup_{x\in\cG} \tilde{\vartheta}_x(0)  \frac{\delta_n}{\inf_{x\in\cG} F_x(h)} + \sup_{x\in\cG} \sup_{u\le h} |\tilde{\vartheta}'_x(u) | h		\right)  \in \cO\left( R^r \frac{ \delta_n  }{\inf_{x\in\cG} F_x(h)} + h \right).
\end{align*}
Similarly, we find that the first term in \eqref{UniformConvHatPsiEq6} is in $\cO( R^r \delta_n / h )$ and that the third term is in $\cO( R^r \delta_n / \inf_{x\in\cG} F_x(h) )$. This proves that \eqref{UniformConvHatPsiEq6} converges to zero as well as the third term in \eqref{UniformConvHatPsiEq4}. 

Consequently, 
\begin{align*}
		&\sup_{x\in\cG} \hnorm{\Psi(x)} \sup_{x\in\cG} \left| \hat{f}_h(x)-f(x)\right| + \sup_{x\in\cG} \hnorm{\hat{g}_h(x)-g(x)} \\
		&=  \cO\left( \frac{ R_n^{5d/2} (\log |I_n|)^{7} }{ \inf_{x\in\cG(R_n) } F_x(h)\,|I_n|^{1/N \cdot 2/(5d+2) } } \right) + \cO\left(\frac{R_n^r}{h \, |I_n|^{1/N \cdot 2/(5d+2) } } \right)
\end{align*}
This completes the proof.
\end{proof}

Next, we give a result for $\cC$-weakly dependent processes
Therefore, we consider the pseudo-norm on $\cC$ defined by
\begin{align}\label{SemiNorm}
		\norm{g}^{\sim} = \sup_{u,v\in \cS, u\neq v } \frac{\left| g(u)-g(v)\right|}{d(u,v)}
\end{align}
for an element $g: \cS\rightarrow \R$ such that $\norm{g}_{\infty}<\infty$. We assume for the next theorem that the kernel function $K$ is zero at 1. Note that we have in this case for the pseudo-norm $\norm{\cdot}^{\sim}$ that $\norm{ K(h^{-1} d(\cdot,x)) }^{\sim} $ is proportional to $h^{-1}$ (from the reverse triangle inequality).

\begin{theorem}[Uniform convergence under weak spatial dependence conditions]\label{UniformConvHatPsiV}
Let Conditions (\ref{Condition1})-(\ref{Condition8}) and (\ref{Condition10}) be satisfied. Let $(n_k:k\in\N)$ be a sequence in $\N^N$ which converges to infinity. Let $R_n$ be a real-valued sequence which has a limit in $(0,\infty]$ and assume that the bandwidth $h=h_n$ converges to zero such that
\begin{align*}%\label{UniformConvHatPsiVEq0}
		\frac{R_n^{4d} \, (\log |I_n|)^8 }{|I_n|^{1/N \cdot 1/(4d+1)} \inf_{x\in\cG(R_n)} F_x(h)^2 \, h_n }\rightarrow 0.
\end{align*}
Then
$$
	\sup_{x\in \cG(R_n) } \hnorm{ \hat{\Psi}_h(x) - \Psi(x) }  = \cO\left( \frac{R_n^{4d} \, (\log |I_n|)^8 }{|I_n|^{1/N \cdot 1/(4d+1)} \inf_{x\in\cG(R_n)} F_x(h)^2 \, h_n } \right) + \cO\left( h^r \right) \quad a.s.
$$
\end{theorem}
\begin{proof}
The structure of the proof is the same as in Theorem~\ref{UniformConvHatPsi}. We can continue with the decomposition of Collomb from \eqref{UniformConvHatPsiEq1} and it remains to demonstrate that both
\begin{align}\label{EqUniformConvHatPsiV1}
		\sup_{x\in\cG} \hnorm{ \hat{f}_h(x) - \E{\hat{f}_h(x) } }  \rightarrow 0 \;a.s. \text{ and } \sup_{x\in\cB} \hnorm{\Psi(x)} \, \sup_{x\in\cG} \hnorm{ \hat{g}_h(x) - \E{\hat{g}_h(x) } } \rightarrow 0 \;a.s.
\end{align}
with the desired rate. Therefore, we can immediately pass to the first term in \eqref{UniformConvHatPsiEq4}. We merely have to adjust the parameters in the exponential inequalities which are given in Equations~\eqref{UniformConvHatPsiEq4bb} and \eqref{UniformConvHatPsiEq4cc}. The analogue of \eqref{UniformConvHatPsiEq4bb} reads now
\begin{align*}
	&\p\left( \max_{1\le j \le m} \hnorm{ \frac{1}{|I_n|} \sum_{s\in I_n} Y_{s} \frac{ \Kh{d(X_s,v_j)}}{F_{v_j} (h)} - \E{ Y_s \frac{ \Kh{d(X_s,v_j)}}{F_{v_j}(h)}} } 		\ge z\right) \\
	&\le A_1 Q_n \exp\left( A_2 \frac{R_n^d}{\delta_n^d} - A_3 (z^2 |I_n|^{1/N} \inf_{x\in\cG} F_x(h)^2 h )^{1/4}		\right), 
\end{align*}
where we use a $\delta_n$ covering and apply Proposition~\ref{WeaklyCDependentHilbert}. The factor $Q_n$ is negligible. 

The analogue of \eqref{UniformConvHatPsiEq4cc} can be bounded with an application of Proposition~\ref{WeaklyCDependent}
\begin{align*}
		&\p\left( \max_{1\le j\le m} \left| \frac{1}{|I_n|} \sum_{s\in I_n} \Kh{d(X_s,v_j)} - \E{ \Kh{d(X_s,v_j)}} \right|  		\ge z \inf_{x\in\cG} F_{x} (h) / \tilde{Y}^{(0)}_0 \right) \\
		&\le A_1 \exp\left( A_2 \frac{R_n^d}{\delta_n^d} - A_3 \frac{z^2 |I_n|^{1/N} \inf_{x\in\cG} F_x(h)^2 h }{R_n^r}		\right).
\end{align*}

In particular, in both cases $l=0$ and $l=1$
$$
			 \max_{1\le j \le m} \hnorm{ \frac{1}{|I_n|} \sum_{s\in I_n} \tilde{Y}^{(l)}_{s} \frac{ \Kh{d(X_s,v_j)}}{F_{v_j} (h)} - \E{ \tilde{Y}^{(l)}_s \frac{ \Kh{d(X_s,v_j)}}{F_{v_j}(h)}} }  = \cO\left( \frac{ (R(n)/\delta_n)^{4d} (\log|I_n|)^8 }{ |I_n|^{1/N} \inf_{x\in\cG} F_x(h)^2 h } \right)
$$

The analogue of second and the third term in \eqref{UniformConvHatPsiEq4} are of a simpler structure because this time the kernel function is Lipschitz continuous on entire $\R$. So in particular, Equation~\eqref{UniformConvHatPsiEq5} becomes simpler. The analogue of the third term in \eqref{UniformConvHatPsiEq4} is again in
$$
		\cO\left( \frac{R_n^r \delta_n}{h_n}		\right)+\cO\left(	 \frac{R_n^r \delta_n}{ \inf_{x\in\cG} F_{x} (h) }	\right).
$$
We can proceed similar as in the proof of Theorem~\ref{UniformConvHatPsi} and choose $\delta_n = |I_n|^{-1/N\cdot 1/(4d+1)}$. We arrive at the conclusion that both terms in \eqref{EqUniformConvHatPsiV1} converge to zero $a.s.$ at the stated rate.
\end{proof}

We can compare the rates of convergence of the estimate $\hat{\Psi}$ from Theorem~\ref{UniformConvHatPsi} and Theorem~\ref{UniformConvHatPsiV} with the results in \cite{ferraty2004nonparametric}. Here the authors consider the estimator on a compact set $\cK\subseteq\cH$ and assume that the data generating process is a strongly mixing time series with a one-dimensional response variable. The further technical assumptions are quite similar. Therefore, we can compare the two rates in the case where $\cK\subseteq\cG(R)$ and where the lattice process $(X,Y)$ is strongly mixing. We obtain for the estimate $\hat{\Psi}$ which is based on $\cH$-valued spatial response variables a rate of 
$$
	\cO \left( \frac{(\log |I_n|)^7 }{  |I_n|^{1/N \cdot 2/(5d+2)} \inf_{x\in \cK} F_x(h) } \right) + \cO\left( \frac{1}{ (|I_n|^{1/N \cdot 2/(5d+2)} h } \right ) + \cO(h^r)
	$$
	because the radius $R=R_n$ of the set $\cG(R)$ can be chosen as constant. In the special case of time series data $((X_t,Y_t):t=1,\ldots,n)$, where the lattice dimension $N$ is one, the rate simplifies as 
$$
	\cO\left( \frac{ (\log n)^7 }{ n^{2/(5d+2)} \inf_{x\in \cK} F_x(h) } \right) + \cO\left( \frac{1}{ (n^{2/(5d+2)} h } \right )  + \cO(h^r).
$$
	The rate obtained by \cite{ferraty2004nonparametric} is derived under the weaker condition that the one-dimensional response variables only satisfy a moment condition and not an exponential tail condition as in our case for Hilbertian response variables. Their rate is given in terms of a parameter $s$ which characterizes the moment condition, a function which is proportional to our function $\inf_{x\in\cK}  F_x(h)$ and a function $\chi$ which is a bound on the maximum of $\inf_{x\in\cK}  F_x(h)^2$ and the joint small ball probability of $X_t$ and $X_{t'}$, for details see \cite{ferraty2004nonparametric}. The rate is in their case
	$$
			\cO\left( \sqrt{ \frac{\log n}{ n \inf_{x\in\cK}  F_x(h) } } \right) + \cO\left(	\sqrt{ \frac{\log n}{n } \frac{ \chi(h)}{\inf_{x\in \cK} F_x(h)^2} \floor*{\frac{n}{\chi(h)}}^s }	\right) + \cO(h^r).
			$$
Hence, the structure of the rate of convergence is similar to ours, in particular, the third $\cO$-expression is also due to the local approximation of $\Psi(X_t)$ by $\Psi(x)$. It is not unexpected that the rate of the first $\cO$-term is slower in the case of a $\cH$-valued response. 

In the case of a constant radius $R$, we obtain for $\cC$-weakly dependent spatial data a rate of
$$
			\cO\left(		\frac{ (\log |I_n|)^8 }{|I_n|^{1/N\cdot 1/(4d+1)} \inf_{x\in\cK}  F_x(h) ^2 \; h } \right) + \cO(h^r).
$$
Again, this rate is similar to the rate of \cite{ferraty2004nonparametric} (for the special case of time series data). Note that the factor $h$ in the denominator of the first $\cO$-expression is due to the $\norm{\cdot}^\sim$-norm of the scaled kernel function $K_h$. Once more the second $\cO$-expression is due to the local approximation of $\Psi(X_t)$ by $\Psi(x)$.

The dimension of the domain of the functions $\cD$ influences the rate negatively in our case. In the case of functional data as curves, $d=1$ and we have the correction factors $2/7$ resp. $1/5$. If the dimension $d$ is bigger, e.g., if we observe manifolds as functional data, the correction factor is even more pronounced. The reason for this is the increasing number of balls of radius $\delta_n$ which cover the space $\cG(R)$. Furthermore, this covering is w.r.t.\ the norm on the Hilbert space and not w.r.t.\ the pseudo-metric $d$. Note that in the proofs it would be sufficient to use a $\delta_n$-covering w.r.t.\ $d$. However, in order to exploit this, we would have to make further assumptions on $d$. Furthermore, in many applications $d$ is a projection-based pseudo metric. Hence, in a possible extension of the current setting, one could consider the case of a sequence of such pseudo-metrics $d_k$ which tend to the metric induced by $\hnorm{\cdot}$.

To conclude, we shortly discuss the influence of the lattice dimension $N$. We see that the sample $I_n$ does not enter in the denominator with its full size but rather with an effective size, where $|I_n|$ is normalized by the $N$-th root. The technical reason for this behavior is explained in the short remark before Corollary~\ref{CorBernsteinImproved}. It is up to future research whether this factor can be removed under the current assumptions with more sophisticated techniques or whether additional assumptions are necessary.

\appendix
\section{Appendix}
\begin{lemma}[\cite{ibragimov1962some}]\label{IbragimovAlphaMixing}
Let $Z_1,\ldots,Z_n$ be real-valued non-negative random variables each $a.s.$ bounded. Set $\alpha \coloneqq \sup_{s\in \{1,\ldots,n\} } \alpha\left( \sigma( Z_i: i \le k), \sigma( Z_i: i > k) \right)$. Then $ \left| \E{ \prod_{i=1}^n Z_i } - \prod_{i=1}^n \E{Z_i} \right| \le (n-1) \, \alpha\, \prod_{i=1}^n \norm{Z_i}_{\infty}$.
\end{lemma}

%\bibliographystyle{abbrvnat}
%\bibliography{Bibliography}

\end{document}